\documentclass[DIV=14,letterpaper]{scrartcl}

\date{}

\usepackage{tikz}
\usetikzlibrary{calc}
\usetikzlibrary{matrix}
\usepackage{amsmath,amssymb,amsfonts,amsthm}
\usepackage[pdftex]{hyperref}
\usepackage{pdfpages}
\usepackage{graphicx}
\usepackage{subfigure}
\allowdisplaybreaks
\usepackage{algorithm,algpseudocode,float}
\usepackage{mathabx}

\newcommand{\V}[1]{\mbox{\boldmath $ #1 $}}

\def \M{\mathbb{M}}

\newcommand{\bey}{\begin{eqnarray}}
\newcommand{\eey}{\end{eqnarray}}

\newcommand{\beq}{\begin{equation}}
\newcommand{\eeq}{\end{equation}}
\theoremstyle{plain}
\newtheorem{thm}{\hspace{6mm}Theorem}[section]
\newtheorem{pro}{\hspace{6mm}Proposition}[section]
\newtheorem{lem}{\hspace{6mm}Lemma}[section]
\newtheorem{co}{\hspace{6mm}Corollary}[section]
\theoremstyle{definition}

\newtheorem{exam}{\hspace{6mm}Example}[section]
\newtheorem{rem}{\hspace{6mm}Remark}[section]

\title{A grid-overlay finite difference method for the fractional {Laplacian} on arbitrary bounded domains}

\author{
Weizhang Huang\thanks{Department of Mathematics, University of Kansas, Lawrence, Kansas, U.S.A.
{\em whuang@ku.edu}}
\; and
Jinye Shen\thanks{Corresponding author. School of Mathematics, Southwestern University of Finance and Economics,
Chengdu, Sichuan, China. {\em jyshen@swufe.edu.cn}}
}

\begin{document}
\vskip 1cm
\maketitle

\begin{abstract}
A grid-overlay finite difference method is proposed for the numerical approximation of the fractional Laplacian on arbitrary bounded domains.
The method uses an unstructured simplicial mesh and an overlay uniform grid for the underlying domain
and constructs the approximation based on a uniform-grid finite difference approximation and a data transfer
from the unstructured mesh to the uniform grid.
 The method takes full advantages of both uniform-grid finite difference approximation in efficient matrix-vector
multiplication via the fast Fourier transform and unstructured meshes for complex geometries and mesh adaptation.
It is shown that its stiffness matrix is similar to a symmetric and positive definite matrix
and thus invertible if the data transfer has full column rank and positive column sums. Piecewise linear interpolation
is studied as a special example for the data transfer. It is proved that the full column rank and positive column sums of linear interpolation
is guaranteed if the spacing of the uniform grid is smaller than or equal to a positive bound
proportional to the minimum element height of the unstructured mesh.
Moreover, a sparse preconditioner is proposed for the iterative solution of the resulting linear system for the homogeneous Dirichlet problem
of the fractional Laplacian. Numerical examples demonstrate that the new method has similar convergence
behavior as existing finite difference and finite element methods and that the sparse preconditioning is effective.
Furthermore, the new method can readily be incorporated with existing mesh adaptation strategies. Numerical results obtained
by combining with the so-called MMPDE moving mesh method are also presented.
\end{abstract}



\newpage

\section{Introduction}
The fractional Laplacian is a fundamental non-local operator for modeling anomalous dynamics and its numerical
approximation has attracted considerable attention recently; e.g. see \cite{Antil-2022,Huang2016,Lischke-2020} and references therein.
A number of numerical methods have been developed along the lines of various representations of
the fractional Laplacian, such as the Fourier/spectral representation, the singular
integral (including Riemann-Liouville and Caputo) representation, the Gr\"{u}nwald-Letnikov representation,
and the heat semi-group representation.
For examples, methods based on the Fourier/spectral representation include
finite difference (FD) methods \cite{Hao2021,Huang2014,Huang2016,Llic-2005,
Ortigueira2006,Ortigueira2008,YangQianqian2011},
spectral element method \cite{Song2017}, and sinc-based method \cite{Antil-2021}.
Methods based on the singular integral representation include
FD methods \cite{Duo-2018,Ying-2020,Sunjing2021},
finite element methods \cite{Acosta2017,Acosta201701,Ainsworth-2017,Ainsworth-2018,Bonito2019,Faustmann2022,Tian2013},
discontinuous Galerkin methods \cite{Du2019,Du2020},
and spectral method \cite{LiHuiyuan2022};
and FD methods \cite{DuNing2019,Pang-2012,Wang2012}
based on the Gr\"{u}nwald-Letnikov representation.
Loosely speaking, most of the existing FD methods have been constructed on uniform grids, have the advantage
of efficient matrix-vector multiplication via the fast Fourier transform (FFT), but do not work for domains with complex geometries
and have difficulty to incorporate with mesh adaptation.
On the other hand, finite element methods can work for arbitrary bounded domains and are easy to combine with mesh adaptation
but suffer from slowness of matrix-vector multiplication because the stiffness matrix is a full matrix.
A sparse approximation to the stiffness matrix and an efficient multigrid implementation
have been proposed by Ainsworth and Glusa \cite{Ainsworth-2018}.
There exists special effort to apply FD and spectral methods to domains with complex geometries.
For example, Song et al. \cite{Song2017} construct an approximation of the fractional Laplacian
based on the spectral decomposition and the spectral element approximation
of the Laplacian operator on arbitrary domains.
Hao et al. \cite{Hao2021} combine a uniform-grid FD method with the penalty method of \cite{Saito-2015} for non-rectangular domains.

\begin{figure}[ht!]
\centering
\includegraphics[width=0.22\textwidth]{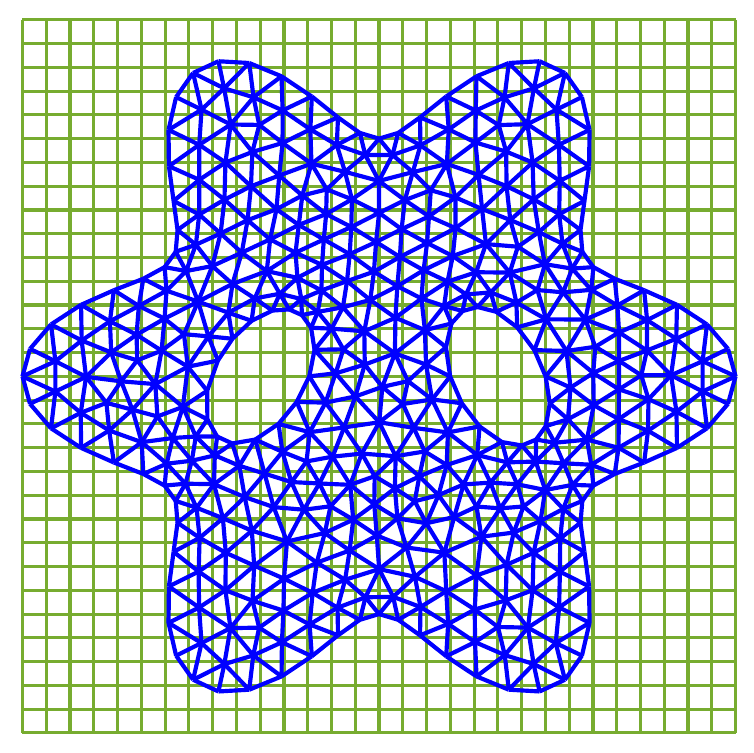}
\caption{A sketch of an unstructured simplicial mesh $\mathcal{T}_h$ (in color blue) overlaid by a uniform grid $\mathcal{T}_{\text{FD}}$
(in color green) .  Boundary value problem \eqref{FL-1} is solved on $\mathcal{T}_h$ that is not necessarily quasi-uniform.}
\label{fig:gridoverlay-1}
\end{figure}

The objective of this work is to present a simple FD method, called the grid-overlay FD method or GoFD,
for the fractional Laplacian on arbitrary bounded domains. The method
has the full advantage of uniform-grid FD methods in efficient matrix-vector multiplication via FFT.
Specifically, we consider the homogeneous Dirichlet problem
\begin{equation}
\label{FL-1}
\begin{cases}
(-\Delta)^{s} u  = f, & \text{ in } \Omega
\\
u = 0, & \text{ in } \Omega^c
\end{cases}
\end{equation}
where $(-\Delta)^{s}$ is the fractional Laplacian with the fractional order $s \in (0,1)$, $\Omega$ is a bounded
domain in $\mathbb{R}^d$ ($d \ge 1$), $\Omega^c \equiv \mathbb{R}^d\setminus \Omega$, and $f$ is a given function.
Given an unstructured simplicial mesh $\mathcal{T}_h$ that fits or approximately fits $\Omega$  and is not necessarily quasi-uniform,
we want to approximate the solution of (\ref{FL-1}) on $\mathcal{T}_h$. A uniform grid
$\mathcal{T}_{\text{FD}}$ (with spacing $h_{\text{\text{FD}}}$) that overlays $\bar{\Omega}$ (see Fig.~\ref{fig:gridoverlay-1}) is first created
and a uniform-grid FD approximation $h_{\text{\text{FD}}}^{-2 s} A_{\text{FD}}$ for the fractional Laplacian is constructed thereon.
Then, the GoFD approximation of the fractional Laplacian on $\mathcal{T}_h$ is defined as
$h_{\text{\text{FD}}}^{-2 s} A_{h}$, where $A_h$ is given by
\begin{equation}
\label{Ah-1}
A_h = D_h^{-1} (I_{h}^{\text{FD}})^T A_{\text{FD}} I_{h}^{\text{FD}}.
\end{equation}
Here, $I_{h}^{\text{FD}}$ is a transfer matrix from $\mathcal{T}_h$ to $\mathcal{T}_{\text{FD}}$ and $D_h$ is the diagonal matrix
formed by the column sums of $I_{h}^{\text{FD}}$.
Notice that the multiplication of $A_h$ with vectors can be performed efficiently
since the multiplication of $A_{\text{FD}}$ with vectors can be carried
out using FFT and $I_{h}^{\text{FD}}$ is sparse.
Moreover, $A_h$ is invertible if $I_{h}^{\text{FD}}$ has full column rank and positive column sums (cf. Theorem~\ref{thm:Ah-1}).
For a special choice of $I_{h}^{\text{FD}}$, piecewise linear interpolation from $\mathcal{T}_h$ to $\mathcal{T}_{\text{FD}}$, Theorem~\ref{thm:IhFD} states that the full column rank and positive column sums of $I_{h}^{\text{FD}}$
are guaranteed if $h_{\text{\text{FD}}}$ is smaller than or equal to a positive bound
proportional to the minimum element height of $\mathcal{T}_h$ (cf. (\ref{hFD-1})).
Stability and sparse preconditioning for the resulting linear system are studied.
Furthermore, the use of unstructured meshes in GoFD allows easy incorporation with existing
mesh adaptation strategies. As an example, the incorporation with
the so-called MMPDE moving mesh method \cite{HK2015,HRR94a,HR11} is discussed.
Numerical examples in 1D, 2D, and 3D are presented to demonstrate that GoFD has similar convergence behavior
as existing FD and finite element methods and that the sparse preconditioning and mesh adaptation are effective.

An outline of the paper is as follows. The construction of uniform-grid FD approximation for the fractional Laplacian is presented in
Section~\ref{SEC:uniformFD}. Section~\ref{SEC:GoFD} is devoted to the description of GoFD, studies of its properties,
and construction of sparse preconditioning. The MMPDE moving mesh method and its combination with GoFD are discussed in
Section~\ref{SEC:MMPDE}. Numerical examples are presented in Section~\ref{SEC:numerics}. Finally,
conclusions are drawn and further comments are given in Section~\ref{SEC:conclusions}.

\section{Uniform-grid FD approximation of the fractional Laplacian}
\label{SEC:uniformFD}

In this section we briefly describe the FD approximation of the fractional Laplacian on a uniform grid through
the Fourier transform. The reader is referred to, e.g., \cite{Hao2021,Huang2016,Ortigueira2006,Ortigueira2008}, for detail.
The properties of the approximation and the computation of the stiffness matrix
and its multiplication with vectors through the fast Fourier transform (FFT) are also discussed.
For notational simplicity and without loss of generality, we restrict our discussion in 2D.

\subsection{FD approximation on a uniform grid}

Consider an absolutely integrable function $u$ in $\mathbb{R}^2$. Recall that its Fourier transform
is defined as
\begin{equation}
\label{FT-1}
\hat{u}(\xi, \eta) = \int_{-\infty}^{\infty} \int_{-\infty}^{\infty}  u(x,y) e^{-i x \xi} e^{-i y \eta} d x d y ,
\end{equation}
and the inverse Fourier transform is given by
\begin{equation}
\label{FT-2}
u(x, y) = \frac{1}{(2\pi)^2} \int_{-\infty}^{\infty} \int_{-\infty}^{\infty} \hat{u}(\xi,\eta) e^{i x \xi} e^{i y \eta} d \xi d \eta .
\end{equation}
Applying the Laplacian operator to the above equation, we have
\[
(-\Delta) u(x, y) = \frac{1}{(2\pi)^2} \int_{-\infty}^{\infty} \int_{-\infty}^{\infty} (\xi^2+\eta^2)
\hat{u}(\xi,\eta) e^{i x \xi} e^{i y \eta} d \xi d \eta .
\]
This implies
\[
\widehat{(-\Delta)u}(\xi, \eta) = (\xi^2+\eta^2) \hat{u}(\xi,\eta) .
\]
Based on this, the Fourier transform of the fractional Laplacian can be defined as
\begin{align}
\label{FL-3}
\widehat{(-\Delta)^s u}(\xi,\eta) = (\xi^2+\eta^2)^{s} \hat{u}(\xi,\eta) .
\end{align}
Accordingly, the fractional Laplacian is given by
\begin{align}
\label{FL-2}
(-\Delta)^{s} u (x,y) = \frac{1}{(2\pi)^2} \int_{-\infty}^{\infty} \int_{-\infty}^{\infty}  \widehat{(-\Delta)^s u}(\xi,\eta)
e^{i x \xi} e^{i y \eta} d \xi d \eta .
\end{align}

A uniform-grid FD approximation for the fractional Laplacian can be defined in a similar manner. Consider an infinite uniform grid (lattice)
\[
(x_j, y_k) = (j h_{\text{\text{FD}}}, k h_{\text{\text{FD}}}),\; j, k \in \mathbb{Z},
\]
where $h_{\text{\text{FD}}}$ is a given positive number.
The discrete Fourier transform (DFT) on this grid is defined as
\begin{align}
\label{FT-3}
\check{u}(\xi,\eta) = \sum_{j=-\infty}^{\infty} \sum_{k=-\infty}^{\infty} u_{j,k} e^{-i x_j \xi} e^{-i y_k \eta} ,
\end{align}
where $u_{j,k} = u(x_j,y_k)$. Here, we use $\check{u}$ to denote the DFT of $u$ to avoid confusion with
the continuous Fourier transform (cf. (\ref{FT-1})). The inverse DFT is given by
\begin{align}
u(x_j,y_k) & = \frac{h_{\text{\text{FD}}}^2}{(2\pi)^2}
\int_{-\frac{\pi}{h_{\text{\text{FD}}}}}^{\frac{\pi}{h_{\text{\text{FD}}}}} \int_{-\frac{\pi}{h_{\text{\text{FD}}}}}^{\frac{\pi}{h_{\text{\text{FD}}}}}
\check{u}(\xi,\eta) e^{i x_j \xi} e^{i y_k \eta} d \xi d \eta
\notag \\
& = \frac{1}{(2\pi)^2} \int_{-\pi}^{\pi} \int_{-\pi}^{\pi}
\check{u}(\frac{\xi}{h_{\text{\text{FD}}}},\frac{\eta}{h_{\text{\text{FD}}}}) e^{i j \xi} e^{i k \eta} d \xi d \eta .
\label{FT-4}
\end{align}
Now we consider a central FD approximation to the Laplacian on $\{ (x_j,y_k)\}$,
\begin{equation}
\label{5-point}
(-\Delta_h) u (x_j, y_k) = \frac{1}{h_{\text{\text{FD}}}^2} (u_{j+1,k} - 2 u_{j,k} + u_{j-1,k})
+ \frac{1}{h_{\text{\text{FD}}}^2} (u_{j,k+1} - 2 u_{j,k} + u_{j,k-1}) .
\end{equation}
Applying the DFT (\ref{FT-3}) to the above equation, we get
\begin{align*}
\widecheck{(-\Delta_h)u} (\xi, \eta) & = \frac{1}{h_{\text{\text{FD}}}^2} \left (4 \sin^{2}(\frac{\xi h_{\text{\text{FD}}}}{2}) + 4 \sin^{2}(\frac{\eta h_{\text{\text{FD}}}}{2})\right )
\check{u}(\xi,\eta) .
\end{align*}
From this, the DFT of the FD approximation of the fractional Laplacian can be defined as
\begin{align}
\label{FL-4}
\widecheck{(-\Delta_h)^s u} (\xi, \eta) & = \frac{1}{h_{\text{\text{FD}}}^{2s}}
\left (4 \sin^{2}(\frac{\xi h_{\text{\text{FD}}}}{2}) + 4 \sin^{2}(\frac{\eta h_{\text{\text{FD}}}}{2})\right )^{s} \check{u}(\xi,\eta) .
\end{align}
Then, the FD approximation of the fractional Laplacian reads as
\begin{align}
(-\Delta_h)^s u(x_j,y_k) & = \frac{1}{ (2\pi)^2}
\int_{-\pi}^{\pi} \int_{-\pi}^{\pi}
\widecheck{(-\Delta_h)^s u} (\frac{\xi}{h_{\text{\text{FD}}}},\frac{\eta}{h_{\text{\text{FD}}}}) e^{i j \xi} e^{i k \eta} d \xi d\eta
\notag \\
& = \frac{1}{(2\pi)^2 h_{\text{\text{FD}}}^{2s}} \sum_{m=-\infty}^{\infty} \sum_{n=-\infty}^{\infty} u_{m,n} \int_{-\pi}^{\pi} \int_{-\pi}^{\pi}  \psi(\xi,\eta) e^{i (j-m) \xi} e^{i (k-n) \eta} d \xi d\eta ,
\label{FL-5}
\end{align}
where
\begin{align}
\label{FL-6}
\psi(\xi,\eta) = \left (4 \sin^{2}(\frac{\xi}{2}) + 4 \sin^{2}(\frac{\eta}{2})\right )^{s} .
\end{align}
If we define
\begin{align}
\label{T-1}
& T_{p,q} = \frac{1}{(2\pi)^2} \int_{-\pi}^{\pi} \int_{-\pi}^{\pi}  \psi(\xi,\eta) e^{i p \xi} e^{i q \eta} d \xi d\eta,
\\
\label{A-1}
& A_{(j,k),(m,n)}  = T_{j-m,k-n} ,
\end{align}
we can rewrite (\ref{FL-5}) into
\begin{align}
\label{FL-8}
(-\Delta_h)^s u(x_j,y_k) = \frac{1}{h_{\text{\text{FD}}}^{2s}}
\sum_{m=-\infty}^{\infty} \sum_{n=-\infty}^{\infty} A_{(j,k),(m,n)} u_{m,n},\qquad -\infty < j, k < \infty .
\end{align}
Notice that $A_{(j,k),(m,n)}$ is the entry of an infinite matrix $A$ at the $(j,k)$-th row and the $(m,n)$-th column
with the understanding that the 2D indices $(j,k)$ and $(m,n)$ are converted into linear indices
using a certain ordering (such as the natural ordering).
Moreover,  $T_{p,q}$'s are the coefficients of the Fourier series of $\psi(\xi,\eta)$. From (\ref{T-1}), it is not difficult
to show
\begin{equation}
\label{T-2}
T_{-p,-q} = T_{p,q}, \quad T_{-p,q} = T_{p,q}, \quad T_{p,-q} = T_{p,q} .
\end{equation}
Furthermore, (\ref{A-1}) implies that $A$ is a Toeplitz matrix of infinite order in 1D and a block Toeplitz matrix of Toeplitz blocks
in multi-dimensions.

\begin{lem}[Parseval's equality]
\label{lem:parseval}
If $\sum_{j=-\infty}^{\infty} \sum_{k=-\infty}^{\infty} u_{j,k}^2 < \infty$, then
\begin{align}
\label{lem:parseval-1}
\frac{1}{ (2\pi)^2} \int_{-\pi}^{\pi} \int_{-\pi}^{\pi} |\check{u}(\frac{\xi}{h_{\text{\text{FD}}}}, \frac{\eta}{h_{\text{\text{FD}}}}) |^2 d \xi d\eta
= \sum_{j=-\infty}^{\infty} \sum_{k=-\infty}^{\infty} u_{j,k}^2 .
\end{align}
\end{lem}
\begin{proof}
(\ref{lem:parseval-1}) can be proved using (\ref{FT-4}) and direct calculation.
\end{proof}

Next, we consider functions with compact support in $\Omega$ for solving the homogeneous Dirichlet problem (\ref{FL-1}).
More specifically, we consider a square $(-R,R) \times (-R,R)$, where $R$ is a positive number such that
$\Omega$ is covered by the square (cf. Fig.~\ref{fig:gridoverlay-1}).
For a given positive integer $N$, we choose $h_{\text{\text{FD}}} = \frac{R}{N}$.
Let $\mathcal{T}_{\text{FD}} = \{ (x_j, y_k), \; -N \le j, k \le N\}$ (a finite unifrom grid) and
$\vec{u}_{\text{FD}} = \{ u_{j,k}, \; j, k = -N, ..., N\}$.
Notice that the right-hand side of (\ref{lem:parseval-1}) becomes a finite double sum for the current situation
since $u_{j,k} = 0$ for any $(x_j,y_k) \notin (-R,R) \times (-R,R)$.

For notational convenience, we denote the restriction of the infinite matrix $A$ on $\mathcal{T}_{\text{FD}}$
by $A_{\text{FD}}$, i.e.,
\begin{equation}
\label{AFD-1}
A_{\text{FD}} = \left ( \frac{}{} A_{(j,k),(m,n)} \right )_{(2N+1)^2\times (2N+1)^2}.
\end{equation}
Notice that $h_{\text{\text{FD}}}^{-2s} A_{\text{FD}}$ is an FD approximation matrix of the fractional Laplacian on $\mathcal{T}_{\text{FD}}$.
Recall from (\ref{A-1}) (now with $-N \le j,k,m,n \le N$) that $A_{\text{FD}}$ is a block Toeplitz matrix of Toeplitz blocks.

\begin{lem}[The fractional Poincar\'{e} inequality]
\label{lem:poincare}
For any function $u$ with support in $(-R,R) \times (-R,R)$, there holds
\begin{align}
\label{lem:poincare:1}
& \int_{-\pi}^{\pi} \int_{-\pi}^{\pi} \big (|\xi|^2 + |\eta|^2\big )^{s} |\check{u}(\frac{\xi}{h_{\text{\text{FD}}}}, \frac{\eta}{h_{\text{\text{FD}}}})|^2 d \xi d\eta
\\
& \qquad \qquad \ge C\, h_{\text{\text{FD}}}^{2s} \int_{-\pi}^{\pi} \int_{-\pi}^{\pi} |\check{u}(\frac{\xi}{h_{\text{\text{FD}}}}, \frac{\eta}{h_{\text{\text{FD}}}})|^2 d \xi d\eta  ,
\notag
\end{align}
where $C$ is a positive constant independent of $u$, $h_{\text{\text{FD}}}$, and $N$.
\end{lem}
\begin{proof}
The proof of the continuous fractional Poincar\'{e} inequality can be found, for example, in \cite[Proposition 1.55, p. 39]{Bahouri-2011}.
(\ref{lem:poincare:1}) is different from the continuous one since it uses DFT instead of the continuous
Fourier transform. Nevertheless, it can be proved by following the proof of the continuous version and using Lemma~\ref{lem:parseval}.
\end{proof}

We remark that the left-hand and right-hand sides of (\ref{lem:poincare:1}) can be viewed as the $H^{2s}$ semi-norm
and $L^2$ norm of $u$, respectively.

\begin{pro}
\label{pro:AFD}
The matrix $A_{\text{FD}} =  \big (A_{(j,k),(m,n)}\big )$ is symmetric and positive definite. Particularly,
\begin{equation}
\vec{u}_{\text{FD}}^T A_{\text{FD}} \vec{u}_{\text{FD}} \ge C\, h_{\text{\text{FD}}}^{2s} \; \vec{u}_{\text{FD}}^T \vec{u}_{\text{FD}},
\quad \forall \vec{u}_{\text{FD}} \in \mathbb{R}^{(2N+1)^2}
\label{pro:AFD-1}
\end{equation}
where $C$ is a positive constant.
\end{pro}

\begin{proof}
From (\ref{A-1}) and (\ref{T-2}), we can see that $A_{\text{FD}}$ is symmetric. Moreover,  from (\ref{T-1}) and (\ref{A-1}) we have
\begin{align*}
&\vec{u}_{\text{FD}}^T A_{\text{FD}} \vec{u}_{\text{FD}}
 = \sum_{j=-N}^N \sum_{k=-N}^N \sum_{m=-N}^N \sum_{n=-N}^N A_{(j,k),(m,n)} u_{j,k} u_{m,n}
\\
 =& \frac{1}{(2\pi)^2} \int_{-\pi}^{\pi} \int_{-\pi}^{\pi} \psi(\xi,\eta)
\cdot \sum_{j=-N}^N \sum_{k=-N}^N u_{j,k} e^{i j \xi+i k \eta}
\cdot \sum_{m=-N}^N \sum_{n=-N}^N u_{m,n} e^{-i m \xi-i n \eta} d \xi d\eta
\\
 =& \frac{1}{(2\pi)^2} \int_{-\pi}^{\pi} \int_{-\pi}^{\pi}  \psi(\xi,\eta)
|\check{u}(\frac{\xi}{h_{\text{\text{FD}}}}, \frac{\eta}{h_{\text{\text{FD}}}})|^2 d \xi d \eta .
\end{align*}
Since
\[
\psi(\xi, \eta) = \left (4 \sin^{2}(\frac{\xi}{2}) + 4 \sin^{2}(\frac{\eta}{2})\right )^{s}
\ge \left (\frac{2}{\pi}\right )^2 \big (|\xi|^2 + |\eta|^2\big )^{s},\quad \forall \xi, \eta \in (-\pi, \pi)
\]
we get
\[
\vec{u}_{\text{FD}}^T A_{\text{FD}} \vec{u}_{\text{FD}} \ge \frac{1}{(2\pi)^2} \left (\frac{2}{\pi}\right )^2 \int_{-\pi}^{\pi} \int_{-\pi}^{\pi}
\big (|\xi|^2 + |\eta|^2\big )^{s}
|\check{u}(\frac{\xi}{h_{\text{\text{FD}}}}, \frac{\eta}{h_{\text{\text{FD}}}})|^2 d \xi d \eta .
\]
Combining this with Lemmas~\ref{lem:parseval} and \ref{lem:poincare} we obtain (\ref{pro:AFD-1}), which implies that $A_{\text{FD}}$
is positive definite.
\end{proof}

When $\Omega = (-R, R)\times (-R, R)$, an FD approximation of the homogeneous Dirichlet problem (\ref{FL-1}) on $\mathcal{T}_{\text{FD}}$
is given by
\begin{equation}
\frac{1}{h_{\text{\text{FD}}}^{2s}} A_{\text{FD}} \vec{u}_{\text{FD}} = \vec{f}_{\text{FD}},
\label{FD-1}
\end{equation}
where $\vec{f}_{\text{FD}} = \{ f(x_j, y_k),\; j, k = -N, ..., N\}$.

\begin{pro}[Stability]
\label{pro:FD-stab}
The solution of (\ref{FD-1}) satisfies
\begin{align}
\label{pro:FD-stab:1}
\vec{u}_{\text{FD}}^T A_{\text{FD}} \vec{u}_{\text{FD}} & \le C_1\, h_{\text{\text{FD}}}^{2s} \, \vec{f}_{\text{FD}}^T \vec{f}_{\text{FD}} ,
\\
\label{pro:FD-stab:2}
\vec{u}_{\text{FD}}^T \vec{u}_{\text{FD}} & \le C_2\, \vec{f}_{\text{FD}}^T \vec{f}_{\text{FD}} ,
\end{align}
where $C_1$ and $C_2$ are constants.
\end{pro}

\begin{proof}
Multiplying $\vec{u}_{\text{FD}}^T$ with (\ref{FD-1}) from left and using the Cauchy-Schwarz inequality, we have
\begin{align*}
\vec{u}_{\text{FD}}^T A_{\text{FD}} \vec{u}_{\text{FD}} & = h_{\text{\text{FD}}}^{2s} \vec{u}_{\text{FD}}^T \vec{f}_{\text{FD}}
= h_{\text{\text{FD}}}^{2s} ( A_{\text{FD}}^{\frac{1}{2}} \vec{u}_{\text{FD}})^T ( A_{\text{FD}}^{-\frac{1}{2}} \vec{f}_{\text{FD}} )
\nonumber\\
& \le h_{\text{\text{FD}}}^{2s} ( \vec{u}_{\text{FD}}^T A_{\text{FD}} \vec{u}_{\text{FD}} )^{\frac 12}
( \vec{f}_{\text{FD}}^T A_{\text{FD}}^{-1} \vec{f}_{\text{FD}} )^{\frac 12} .
\end{align*}
Thus, we have
\[
\vec{u}_{\text{FD}}^T A_{\text{FD}} \vec{u}_{\text{FD}} \le h_{\text{\text{FD}}}^{4s} \vec{f}_{\text{FD}}^T A_{\text{FD}}^{-1} \vec{f}_{\text{FD}} .
\]
Notice that Proposition~\ref{pro:AFD} implies $\| A_{\text{FD}}^{-1}  \| \le C h_{\text{\text{FD}}}^{-2 s}$.
Combining these results gives (\ref{pro:FD-stab:1}).

The inequality (\ref{pro:FD-stab:2}) follows from (\ref{pro:FD-stab:1}) and Proposition~\ref{pro:AFD}.
\end{proof}

\begin{rem}
\label{rem:FD-error}
Error estimates and convergence order for the FD approximation of the fractional Laplacian
have been established for sufficiently smooth solutions by a number of researchers;
e.g., see \cite{Duo-2018,Hao2021,Huang2016}.
Unfortunately, the existing analysis does not apply to solutions of (\ref{FL-1}) with the optimal regularity
$H^{s + 1/2-\epsilon}(\Omega)$ for any $\epsilon >0$ \cite{Borthagaray-2018}.
Nevertheless, for those solutions it has been observed numerically (e.g., see \cite{Duo-2018,Hao2021})
that the FD approximation converges at $\mathcal{O}(h_{\text{\text{FD}}}^{s})$ in $L^\infty$ norm.
This is confirmed by our numerical results (cf. Example~\ref{ex:5}).
Our results also suggest that the error is $\mathcal{O}(h_{\text{\text{FD}}}^{\min (1,s + 1/2)})$ in $L^2$ norm.
\qed
\end{rem}

\subsection{Computation of the multiplication of $A_{\text{FD}}$ with vectors using FFT}
\label{SEC:AU-FFT}

For the moment, we assume that $T_{p,q}$'s have been computed.
The process of computing  the multiplication of $A_{\text{FD}}$ with vectors
starts with computing the DFT of $T_{p,q}$'s, i.e.,
\[
\check{T}_{m,n} = \sum_{p=-2N}^{2N-1} \sum_{q=-2N}^{2N-1} T_{p,q} e^{-\frac{i 2 \pi (m+2N)(p+2N)}{4N}}
e^{-\frac{i 2 \pi (n+2N)(q+2N)}{4N}},
\]
for $m,n = - 2N, ..., 2N-1$.
The inverse DFT is
\[
T_{p,q} = \frac{1}{(4N)^2} \sum_{m=-2N}^{2N-1} \sum_{n=-2N}^{2N-1} \check{T}_{m,n} e^{\frac{i 2 \pi (m+2N)(p+2N)}{4N}}
e^{\frac{i 2 \pi (n+2N)(q+2N)}{4N}},
\]
for $p,q = - 2N, ..., 2N-1$.
Then, from (\ref{A-1}) we have
\begin{align}
(A_{\text{FD}} \vec{u}_{\text{FD}})_{(j,k)}
 & = \sum_{m=-N}^N \sum_{n=-N}^N T_{j-m,k-n} u_{m,n}
\notag \\
 =& \sum_{m=-N}^N \sum_{n=-N}^N u_{m,n}
\frac{1}{(4N)^2} \sum_{p=-2N}^{2N-1} \sum_{q=-2N}^{2N-1} \check{T}_{p,q} e^{\frac{i 2 \pi (p+2N)(j-m+2N)}{4N}}
e^{\frac{i 2 \pi (q+2N)(k-n+2N)}{4N}}
\notag \\
 =& \frac{1}{(4N)^2} \sum_{p=-2N}^{2N-1} \sum_{q=-2N}^{2N-1} \check{T}_{p,q} (-1)^{p+2N} (-1)^{q+2N}
e^{\frac{i 2 \pi (p+2N)(j+N)}{4N}}
e^{\frac{i 2 \pi (q+2N)(k+N)}{4N}}
\notag \\
& \qquad \cdot \sum_{m=-N}^N \sum_{n=-N}^N u_{m,n}
  e^{-\frac{i 2 \pi (p+2N)(m+N)}{4N}}
e^{-\frac{i 2 \pi (q+2N)(n+N)}{4N}} .
\label{AU-1}
\end{align}
If we expand $u$ into
\[
\tilde{u}_{m,n} = \begin{cases} u_{m,n}, & \text{ for } -N \le m,n \le N \\ 0, & \text{ otherwise} \end{cases}
\]
we can rewrite the inner double sum in (\ref{AU-1}) into
\begin{align*}
&\sum_{m=-N}^N \sum_{n=-N}^N u_{m,n}
  e^{-\frac{i 2 \pi (p+2N)(m+N)}{4N}} e^{-\frac{i 2 \pi (q+2N)(n+N)}{4N}}
\\
& \qquad \qquad =   \sum_{m=-N}^{3N-1} \sum_{n=-N}^{3 N-1} \tilde{u}_{m,n}
e^{-\frac{i 2 \pi (p+2N)(m+N)}{4N}} e^{-\frac{i 2 \pi (q+2N)(n+N)}{4N}},
\end{align*}
which is the DFT of $\tilde{u}$ (denoted by $\check{\tilde{u}}$). Then,
\begin{align*}
(A_{\text{FD}}\vec{u}_{\text{FD}})_{(j,k)}
 = \frac{1}{(4N)^2} \sum_{p=-2N}^{2N-1} \sum_{q=-2N}^{2N-1} \check{T}_{p,q} \check{\tilde{u}}_{p,q}
(-1)^{p+2N} (-1)^{q+2N} e^{\frac{i 2 \pi (p+2N)(j+N)}{4N}} e^{\frac{i 2 \pi (q+2N)(k+N)}{4N}} .
\end{align*}
Thus, $A_{\text{FD}}\vec{u}_{\text{FD}}$ can be obtained as the inverse DFT of $\check{T}_{p,q} \check{\tilde{u}}_{p,q}
(-1)^{p+2N} (-1)^{q+2N}$.

In summary, $A_{\text{FD}}\vec{u}_{\text{FD}}$ can be carried out by performing three FFTs. From the complexity of FFT,
we can estimate the cost of computing $A_{\text{FD}}\vec{u}_{\text{FD}}$ (in $d$-dimensions) as
$
\mathcal{O}(N^d \log (N^d)) .
$
It is worth mentioning that the number of grid points of $\mathcal{T}_{\text{FD}}$ and length of vector $\vec{u}_{\text{FD}}$
are $\mathcal{O}(N^d)$. Thus, the cost is almost linear about the number of grid points.

\subsection{Computation of matrix $T$}
\label{SEC:T}

We rewrite (\ref{T-1}) into
\begin{align}
\label{T-0}
T_{p,q} = (-1)^{p+q} \int_{0}^{1} \int_{0}^{1}  \tilde{\psi}(\xi,\eta) e^{i 2 \pi p \xi} e^{i 2 \pi q \eta} d \xi d\eta,
\end{align}
where
\begin{equation}
\label{psi-2}
\tilde{\psi}(\xi,\eta) = \big ( 4 \cos^2(\pi \xi) + 4 \cos^2(\pi \eta) \big )^{s} .
\end{equation}
Recall from (\ref{T-2}) that we only need to compute $T_{p,q}$ for $ 0 \le p, q\le 2N$.
We first consider the composite trapezoidal rule (2rd-order). For any given integer $M \ge  2N+1$, let
\[
\xi_j = \frac{j}{M}, \quad \eta_k = \frac{k}{M}, \quad j, k = 0, 1, ..., M.
\]
Then,
\begin{align}
T_{p,q} & = (-1)^{p+q} \sum_{j=0}^{M-1} \sum_{k=0}^{M-1}
\int_{\xi_{j}}^{\xi_{j+1}}\int_{\eta_{k}}^{\eta_{k+1}}
  \tilde{\psi}(\xi,\eta) \cdot e^{i 2 \pi p \xi} e^{i 2 \pi q\eta}d\xi d \eta
\notag \\
& \approx \frac{(-1)^{p+q}}{M^2} \sum_{j=0}^{M-1} \sum_{k=0}^{M-1}
  \tilde{\psi}(\xi_{j},\eta_{k}) \cdot e^{\frac{i 2 \pi p j}{M}} e^{\frac{i 2 \pi q k}{M}} ,
  \quad 0 \le p, q \le 2 N.
\label{T-3}
\end{align}
Thus, $T_{p,q}$'s can be obtained with the inverse FFT.

It should be pointed out that $M$ should be chosen much larger than $2N+1$ in the above procedure for accurate computation of $T$
due to the highly oscillatory nature of the factor $e^{i 2 \pi p \xi} e^{i 2 \pi q \eta}$ in (\ref{T-0}).
As a result, the computation of $T$ can be very expensive in terms of CPU time and memory for large $N$.
To avoid this difficulty, we use Filon's approach \cite{Filon-1928} designed
for highly oscillatory integrals including (\ref{T-0}). We explain this in 1D,
\begin{align*}
T_{p} & = (-1)^{p} \int_{0}^{1}  \tilde{\psi}(\xi) e^{i 2 \pi p \xi}  d \xi
= (-1)^{p} \sum_{j = 0}^{M-1} \int_{\xi_{j}}^{\xi_{j+1}}  \tilde{\psi}(\xi) e^{i 2 \pi p \xi} d \xi .
\end{align*}
The key idea of Filon's approach is to approximate $\tilde{\psi}(\xi)$ with a polynomial on each subinterval and then carry out
the resulting integrals analytically. Here, we approximate $\tilde{\psi}(\xi)$ with a linear polynomial
on each subinterval. In this situation, we can first perform integration by parts and then do the approximation, i.e.,
\begin{align*}
T_{p} & = (-1)^{p} \sum_{j = 0}^{M-1} \left ( \frac{1}{i 2 \pi p} \tilde{\psi}(\xi) e^{i 2 \pi p \xi} |_{\xi_{j}}^{\xi_{j+1}}
- \frac{1}{i 2 \pi p} \int_{\xi_{j}}^{\xi_{j+1}} \tilde{\psi}'(\xi) e^{i 2 \pi p \xi} d \xi \right ).
\end{align*}
The sum of the first term in the bracket vanishes since $\tilde{\psi}(\xi) e^{i 2 \pi p \xi}$ is periodic.
Then,
\begin{align}
T_{p} & = \frac{(-1)^{p+1}}{i 2 \pi p} \sum_{j = 0}^{M-1} \int_{\xi_{j}}^{\xi_{j+1}} \tilde{\psi}'(\xi) e^{i 2 \pi p \xi} d \xi
\notag \\
&\approx \frac{(-1)^{p+1}}{i 2 \pi p} \sum_{j = 0}^{M-1} \int_{\xi_{j}}^{\xi_{j+1}}
\frac{\tilde{\psi}(\xi_{j+1})-\tilde{\psi}(\xi_{j})}{1/M} e^{i 2 \pi p \xi} d \xi
\notag \\
& = \frac{(-1)^{p+1}}{(i 2 \pi p)^2} \sum_{j = 0}^{M-1}  \frac{\tilde{\psi}(\xi_{j+1})-\tilde{\psi}(\xi_{j})}{1/M}
\left (e^{i 2 \pi p \xi_{j+1}} - e^{i 2 \pi p \xi_{j}}\right ) .
\notag
\end{align}
Thus,
\begin{align}
T_{p} & \approx
\frac{(-1)^{p+1}}{(2 \pi p)^2} \sum_{j = 0}^{M-1}  \frac{\tilde{\psi}(\xi_{j+1})-2 \tilde{\psi}(\xi_{j})+\tilde{\psi}(\xi_{j-1})}{1/M}
e^{\frac{i 2 \pi p j}{M}} , \quad p = 0, ..., 2N .
\label{Tp-1D-2}
\end{align}
Notice that (\ref{Tp-1D-2}) can be computed using
FFT. Moreover, the error of the quadrature is $\mathcal{O}(\frac{1}{M^2})$, independent of $N$.

In our computation, we combine the Filon approach with Richardson's extrapolation.
We use $M = 2^{14}$ and $2^{11}$ for the finest level of Richardson's extrapolation for 2D and 3D computation, respectively.
For 1D computation, we use the analytical formula (e.g., see \cite{Ortigueira2008}),
\begin{equation}
\label{T-1D-0}
T_p =  \frac{(-1)^p \Gamma(2s+1)}{\Gamma(p + s + 1) \Gamma(s-p+1)}, \quad p = 0, ..., 2N
\end{equation}
where $\Gamma(\cdot )$ is the $\Gamma$-function.

It is interesting to point out that the cost of computing the matrix $T$ via FFT is $\mathcal{O}(M^d \log (M^d))$.
Assuming that a computer can perform $10^9$ floating-point operations per second, computing the matrix $T$ takes about
$(M^d \log (M^d))/10^9 = 5.2$ seconds in 2D (with $M = 2^{14}$) and $196$ seconds in 3D (with $M = 2^{11}$).

\section{The grid-overlay FD method}
\label{SEC:GoFD}

In this section we describe GoFD for solving (\ref{FL-1}) in $d$-dimensions ($d\ge 1$) on arbitrary bounded domain $\Omega$.
We also study the choice of $h_{\text{\text{FD}}}$ that guarantees
the column full rank and positive column sums of the transfer matrix and therefore the solvability of the linear system resulting from
the GoFD discretization of (\ref{FL-1}). Furthermore, the iterative solution and sparse preconditioning for the linear system are discussed.

\subsection{GoFD for arbitrary bounded domains}

For a given bounded domain $\Omega \in \mathbb{R}^d$, we assume that an unstructured simplicial mesh $\mathcal{T}_h$
has been given that fits or approximately fits $\partial \Omega$.
Then, we take a $d$-dimensional cube, $(-R,R)^d$, such that it covers
$\Omega$ and $\mathcal{T}_h$. An overlaying uniform grid (denoted by $\mathcal{T}_{\text{FD}}$)
is created with $2N+1$ nodes in each axial direction for some positive integer $N$ and the spacing is given by
\[
h_{\text{\text{FD}}} = \frac{R}{N}.
\]
See Fig.~\ref{fig:gridoverlay-1} for a sketch of $\mathcal{T}_{\text{FD}}$ and $\mathcal{T}_h$.

Then, a uniform-grid FD approximation $h_{\text{\text{FD}}}^{-2s} A_{\text{FD}}$ (of size $(2N+1)^d\times (2N+1)^d$)
of the fractional Laplacian can be obtained on $\mathcal{T}_{\text{FD}}$ as described in the previous section.
We define $h_{\text{\text{FD}}}^{-2s} A_h$ as the GoFD approximation of the fractional Laplacian on $\Omega$,
where $A_h$ is defined in (\ref{Ah-1}).

\begin{rem}
\label{rem:Ah-1}
The matrix $D_h^{-1} (I_{h}^{\text{FD}})^T$ represents a data transfer from grid $\mathcal{T}_{\text{FD}}$ to mesh $\mathcal{T}_h$.
It is taken as a transpose of $I_{h}^{\text{FD}}$ so that $A_h$ is similar to a symmetric and
positive definite matrix (cf. Theorem~\ref{thm:Ah-1} below).
Moreover,  $D_h^{-1}$ is included in the definition so that the row sums of $D_h^{-1} (I_{h}^{\text{FD}})^T$ are equal to one
and, thus, the transfer preserves constant functions. This inclusion is necessary for the data transfer $D_h^{-1} (I_{h}^{\text{FD}})^T$
to be consistent.
\qed
\end{rem}

\begin{thm}
\label{thm:Ah-1}
If $I_{h}^{\text{FD}}$ has full column rank and positive column sums, then $A_h$ defined in (\ref{Ah-1}) is similar
to a symmetric and positive definite matrix, i.e.,
\begin{equation}
\label{Ah-2}
A_h = D_h^{-\frac{1}{2}}  \cdot \big (I_{h}^{\text{FD}} D_h^{-\frac{1}{2}}\big )^T A_{\text{FD}} \big (I_{h}^{\text{FD}} D_h^{-\frac{1}{2}}\big ) \cdot D_h^{\frac{1}{2}}.
\end{equation}
As a consequence, $A_h$ is invertible.
\end{thm}

\begin{proof}
When $I_{h}^{\text{FD}}$ has positive column sums, $D_h$ is invertible and thus, the definition (\ref{Ah-1}) is meaningful. Moreover,
it is straightforward to rewrite (\ref{Ah-1}) into (\ref{Ah-2}), indicating that $A_h$ is similar to
\begin{equation}
\label{Ah-3}
\tilde{A}_h = \big (I_{h}^{\text{FD}} D_h^{-\frac{1}{2}}\big )^T A_{\text{FD}} \big (I_{h}^{\text{FD}} D_h^{-\frac{1}{2}}\big ) .
\end{equation}
It is obvious that $\tilde{A}_h$ is symmetric and positive semi-definite. If we can show it to be nonsingular, then $\tilde{A}_h$ is positive definite.
Assume that there is a vector $\vec{u}_h$ such that
$
\vec{u}_h^T \tilde{A}_h \vec{u}_h = 0.
$
Since $A_{\text{FD}}$ is symmetric and positive definite (Proposition~\ref{pro:AFD}), this implies
$
I_{h}^{\text{FD}} D_h^{-\frac{1}{2}} \vec{u}_h = 0.
$
The full column rank assumption of $I_{h}^{\text{FD}}$ and $D_h$ being diagonal mean $\vec{u}_h = 0$. Thus, $\tilde{A}_h$ is nonsingular and,
therefore, $\tilde{A}_h$ is symmetric and positive definite.
\end{proof}

\begin{rem}
\label{rem:Ah-2}
The full column rank assumption of $I_{h}^{\text{FD}}$ in Theorem~\ref{thm:Ah-1} implies that $I_{h}^{\text{FD}}$
has at least as many rows as columns, i.e. $(2N+1)^d\ge N_{v}$. In other words, $\mathcal{T}_{\text{FD}}$ should have
as many vertices as $\mathcal{T}_h$ does.
\qed
\end{rem}

We now consider a special choice of $I_h^{\text{FD}}$:  linear interpolation.
Denote the vertices of $\mathcal{T}_h$ by $\V{x}_j$, $j = 1, ..., N_v$, and the vertices of $\mathcal{T}_{\text{FD}}$ by
$\V{x}_{k}^{\text{FD}}$, $k = 1, ..., N_v^{\text{FD}}$. Consider the piecewise linear interpolation on $\mathcal{T}_h$,
\begin{equation}
\label{interp-1}
I_h u (\V{x}) = \sum_{j=1}^{N_v} u_j \phi_j(\V{x}),
\end{equation}
where $u_j = u(\V{x}_j)$ and $\phi_j$ is the Lagrange-type linear basis function associated with vertex $\V{x}_j$.
Here we assume that all of the basis functions vanish outside $\bar{\Omega}$.
Recall that
\begin{align*}
0 \le \phi_j(\V{x}) \le 1;\qquad
\phi_j(\V{x}) = 0, \quad \forall \V{x} \notin \bar{\omega}_j;
\qquad
\sum_{j=1}^{N_v} \phi_j(\V{x}) = 1, \quad \forall \V{x} \in \bar{\Omega}
\end{align*}
where $\omega_j$ is the patch of elements that have $\V{x}_j$ as one of their vertices. When restricted on an element
$K$ of $\mathcal{T}_h$, the linear interpolation can be expressed as
\begin{equation}
\label{interp-2}
I_h u|_K (\V{x}) = \sum_{j=0}^{d} u_j^K \phi_j^K(\V{x}), \quad \forall \V{x} \in K .
\end{equation}
Let $\vec{u}_h = \{ u(\V{x}_j), j = 1, ..., N_v\}$. Then, for $k = 1, ..., N_v^{\text{FD}}$,
\begin{equation}
\label{I_h_FD-1}
(I_h^{\text{FD}} \vec{u}_h)_k = I_h u (\V{x}_k^{\text{FD}}) = \sum_{j=1}^{N_v} u_j \phi_j(\V{x}_k^{\text{FD}}) .
\end{equation}
This gives
\begin{equation}
\label{IhFD-1}
(I_h^{\text{FD}})_{k,j} = \phi_j(\V{x}_k^{\text{FD}}), \quad k = 1, ..., N_v^{\text{FD}}, \; j = 1, ..., N_v .
\end{equation}

It is worth pointing out that $I_h$ conforms with the homogeneous Dirichlet boundary condition.
Consider an arbitrary vertex $\V{x}_k^{\text{FD}} \in \mathcal{T}_{\text{FD}}$. When $\V{x}_k^{\text{FD}} \notin \Omega$,
it is obvious that $I_h u (\V{x}_k^{\text{FD}}) = 0$. When $\V{x}_k^{\text{FD}} \in \Omega$, on the other hand, there exists
a mesh element $K \in \mathcal{T}_h$ containing $\V{x}_k^{\text{FD}}$ since $\mathcal{T}_h$ is a boundary-fitted mesh for $\Omega$.
In this case, $I_h u (\V{x}_k^{\text{FD}})$ is determined by the values of $u$ at the vertices of $K$ (cf. (\ref{interp-2})).
Particularly, $I_h u (\V{x}_k^{\text{FD}})$ is affected by the values of $u$ on the boundary of $\Omega$ when $K$ is a boundary element.
In this sense, $I_h$ and thus $I_h^{\text{FD}}$ (cf. (\ref{I_h_FD-1}))
and $A_h$ (cf. (\ref{Ah-1})) feel the homogeneous Dirichlet boundary condition.

In the following analysis, we need the following quantities,
\[
N_{val} = \max\limits_{j=1, ..., N_v} \# \omega_j,\quad
N_{\text{FD}}^h = \max\limits_{K\in \mathcal{T}_h} \# \{ \V{x}_k^{\text{FD}} \in \bar{K} \},
\]
where $\#$ stands for the number of members in a set. $N_{val}$ is usually referred to as the valence of $\mathcal{T}_h$.
$N_{\text{FD}}^h$ can be estimated as
\[
N_{\text{FD}}^h \approx \frac{|K|_{max}}{h_{\text{\text{FD}}}^{d}},
\]
where $|K|_{max}$ denotes the volume of the largest element of $\mathcal{T}_h$.
We assume that both $N_{val}$ and $N_{\text{FD}}^h$ are finite and small.

\begin{lem}
\label{lem:IhFD-1}
The transfer matrix $I_h^{\text{FD}}$ associated with piecewise linear interpolation has the following properties.
\begin{itemize}
\item[(i)] Nonnegativity and boundedness: $0 \le (I_h^{\text{FD}})_{k,j} \le 1$ for all $k,j$.
\item[(ii)] Sparsity: $(I_h^{\text{FD}})_{k,j} = 0$ when $\V{x}_k^{\text{FD}} \notin \omega_j$.
\item[(iii)] The row sums of $I_h^{\text{FD}}$ are either 0 or 1, i.e.,
\[
\sum_{j=1}^{N_v} (I_h^{\text{FD}})_{k,j} = \begin{cases} 1, & \text{ for } \V{x}_k^{\text{FD}} \in \bar{\Omega} \\
 0, & \text{ otherwise} \end{cases}
 \quad \forall k = 1, ..., N_v^{\text{FD}}.
\]
\item[(iv)] The column sums, $D_{h,j} = \sum_{k=1}^{N_v^{\text{FD}}} (I_h^{\text{FD}})_{k,j}$, $j=1, ..., N_v$, are bounded by
\begin{equation}
\label{lem:IhFD-1:1}
\max_{\V{x}_k^{\text{FD}} \in \omega_j} \phi_j(\V{x}_k^{\text{FD}}) \le D_{h,j} \le N_{val} N_{\text{FD}}^h .
\end{equation}
\item[(v)]
\begin{equation}
\label{lem:IhFD-1:2}
\lambda_{max}\big (  (I_h^{\text{FD}})^T I_h^{\text{FD}}\big ) \le N_{val} N_{\text{FD}}^h, \quad
\lambda_{max}\big (  I_h^{\text{FD}} (I_h^{\text{FD}})^T \big ) \le N_{val} N_{\text{FD}}^h .
\end{equation}
\end{itemize}
\end{lem}

\begin{proof}
(i), (ii), and (iii) follow from (\ref{IhFD-1}) and the properties of linear basis functions. For (iv), we have
\[
D_{h,j} = \sum_{k=1}^{N_v^{\text{FD}}} \phi_j(\V{x}_k^{\text{FD}})
= \sum_{\V{x}_k^{\text{FD}} \in \omega_j} \phi_j(\V{x}_k^{\text{FD}})
= \sum_{K \in \omega_j} \sum_{\V{x}_k^{\text{FD}} \in K} \phi_j(\V{x}_k^{\text{FD}}) .
\]
Then, (\ref{lem:IhFD-1:1}) follows from the above equation, the definition
of $N_{val}$ and $N_{\text{FD}}^h$, and the nonnegativity of the basis functions.

For (v), we notice that both $(I_h^{\text{FD}})^T I_h^{\text{FD}}$ and $ I_h^{\text{FD}} (I_h^{\text{FD}})^T$ are nonnegative matrices.
Using (iii) and (iv), we can show that their row sums are bounded above by $N_{val} N_{\text{FD}}^h$. Then,
(\ref{lem:IhFD-1:2}) follows from \cite[Theorem 8.1.22]{HJ1985} about the spectral radius of nonnegative matrices.
\end{proof}

Next we study how small $h_{\text{\text{FD}}}$ should be (or, equivalently, how large $N$ should be)
to guarantee that $D_h$ is invertible and $I_h^{\text{FD}}$ has full column rank.
To this end, we need a few properties of simplexes in $\mathbb{R}^d$.
For a simplex $K$, we denote the facet formed
by all of its vertices except $\V{x}_j^K$ by $S_j$ and the distance (called the height or altitude) from $\V{x}_j^K$ to $S_j$
by $a_j$. The minimum height of $K$ is denoted by $a_K$, i.e., $a_K = \min_{j} a_j$ and the minimum element height
of $\mathcal{T}_h$ is denoted by $a_h$, i.e., $a_h = \min_{K} a_K$.

\begin{lem}
\label{lem:simplex-cube}
Any simplex $K \in \mathbb{R}^d$ contains a cube of side length at least $\frac{2 a_K}{(d+1)\sqrt{d}}$.
\end{lem}

\begin{proof}
It is known from geometry (e.g., see \cite[Theorem 1]{Nevskii-2018}) that the radius of the largest ball inscribed in any simplex $K$ is
related to the heights of $K$ as
\[
\frac{1}{r_{in}} = \sum_{j=1}^{d+1} \frac{1}{a_j} .
\]
From this, we have
\[
r_{in} \ge \frac{a_K}{d+1}.
\]
Since the length of the diagonals of the largest cube inscribed in the ball is equal to
the diameter of the ball, i.e., $\sqrt{d} \, a = 2 r_{in}$, where $a$ is the side length of the cube,
we get
\[
a = \frac{2 r_{in}}{\sqrt{d}} \ge \frac{2 a_K}{(d+1) \sqrt{d}}.
\]
\end{proof}

\begin{lem}
\label{lem:simplex-barycentric}
The $j$-th barycentric coordinate of an arbitrary point $\V{x}$ on $K$, $\phi_j^K(\V{x})$, is equal to the ratio
of the distance from $\V{x}$ to facet $S_j$, to the height $a_j$.
\end{lem}

\begin{proof}
The conclusion follows from $\phi_j^K(\V{x}_j^K) = 1$, $\phi_j^K|_{S_j} = 0$, and
the linearity of $\phi_j^K$.
\end{proof}

\begin{lem}
\label{lem:simplex-interp}
Consider a simplex $\tilde{K} \subset K$ with vertices $\V{y}_k$, $k = 0, ..., d$, and define
\[
v_k = I_h u |_K (\V{y}_k) = \sum_{j=0}^{d} u_j^K \phi_j^K(\V{y}_k), \quad k = 0, ..., d.
\]
Then,
\begin{equation}
\label{lem:simplex-interp-1}
\frac{|\tilde{K}|^2}{(d+1)^{d} |K|^2} \sum_{j=0}^d (u_j^K)^2 \le \sum_{j=0}^d v_j^2
\le (d+1) \sum_{j=0}^d (u_j^K)^2 ,
\end{equation}
where $|K|$ and $|\tilde{K}|$ denote the volume of $K$ and $\tilde{K}$, respectively.
\end{lem}

\begin{proof}
Recalling that $\phi_j^K(\V{x}) \ge 0$ and $\sum_{j=0}^d \phi_j^K(\V{x}) = 1$, we have
\begin{align*}
\sum_{k=0}^d v_k^2 & \le \sum_{k=0}^d \left (\sum_{j=0}^{d} |u_j^K| \phi_j^K(\V{y}_k) \right )^2
 \le \sum_{k=0}^d \left (\sum_{j=0}^{d} |u_j^K|^2 \phi_j^K(\V{y}_k) \right )
 \le (d+1) \sum_{j=0}^{d} |u_j^K|^2,
\end{align*}
which gives the right inequality of (\ref{lem:simplex-interp-1}).

To prove the left inequality of (\ref{lem:simplex-interp-1}), from
\[
\V{x} = \sum_{j=0}^d \V{x}_j^K \phi_j^K(\V{x}), \quad 1 = \sum_{j=0}^d \phi_j^K(\V{x}),
\]
we have
\[
\begin{bmatrix} \phi_0^K \\ \vdots \\ \phi_d^K \end{bmatrix} = E^{-1} \begin{bmatrix} \V{x} \\ 1 \end{bmatrix},
\quad E = \begin{bmatrix} \V{x}_0^K & \V{x}_1^K & \cdots & \V{x}_d^K \\ 1 & 1 & \cdots & 1 \end{bmatrix} .
\]
It is known that $\det(E) = d!|K|$. Using this and letting $\vec{u}^K = (u_0^K, ..., u_d^K)^T$, we have
\begin{align}
\sum_{k=0}^d v_k^2 & = \sum_{k=0}^d (\vec{u}^K)^T E^{-1} \begin{bmatrix} \V{y}_k \\ 1 \end{bmatrix}
\begin{bmatrix} \V{y}_k \\ 1 \end{bmatrix}^T E^{-T} \vec{u}^K
\notag \\
& =  (\vec{u}^K)^T E^{-1} \sum_{k=0}^d \begin{bmatrix} \V{y}_k \\ 1 \end{bmatrix}
\begin{bmatrix} \V{y}_k \\ 1 \end{bmatrix}^T E^{-T} \vec{u}^K
\notag \\
& = (\vec{u}^K)^T E^{-1} B B^T E^{-T} \vec{u}^K,
\label{lem:simplex-interp-2}
\end{align}
where
\[
B = \begin{bmatrix} \V{y}_0 & \V{y}_1 & \cdots & \V{y}_d \\ 1 & 1 & \cdots & 1 \end{bmatrix} .
\]
The right inequality of (\ref{lem:simplex-interp-1}) implies that
\[
\lambda_{max}(E^{-1} B B^T E^{-T}) \le (d+1).
\]
On the other hand,
\[
\det(E^{-1} B B^T E^{-T}) = \frac{\det(B)^2}{\det(E)^2} = \frac{|\tilde{K}|^2}{|K|^2} .
\]
Since the determinant of a matrix is equal to the product of its eigenvalues, we get
\begin{align*}
\lambda_{min}(E^{-1} B B^T E^{-T})
\ge \frac{\det(E^{-1} B B^T E^{-T})}{\lambda_{max}(E^{-1} B B^T E^{-T})^d}
\ge \frac{|\tilde{K}|^2}{(d+1)^{d} |K|^2} .
\end{align*}
Combining this with (\ref{lem:simplex-interp-2}) we obtain the left inequality of (\ref{lem:simplex-interp-1}).
\end{proof}

\begin{thm}
\label{thm:IhFD}
If we choose
\begin{equation}
h_{\text{\text{FD}}} \le \frac{a_h}{(d+1)\sqrt{d}},
\label{hFD-1}
\end{equation}
where $a_h$ is the minimum element height of $\mathcal{T}_h$, then the transfer matrix $I_h^{\text{FD}}$ associated with piecewise
linear interpolation has the following properties.
\begin{itemize}
\item[(i)]
\begin{equation}
\frac{1}{(d+1)\sqrt{d}} \cdot \frac{a_h}{h} \le D_{h,j} \le N_{val} N_{\text{FD}}^h,\quad \forall j = 1, ..., N_v
\label{thm:IhFD:0}
\end{equation}
and thus, $D_h$ is invertible. Here, $h$ is the maximum element diameter of $\mathcal{T}_h$.
\item[(ii)] The minimum eigenvalue of $(I_{h}^{\text{FD}})^T I_{h}^{\text{FD}}$ is bounded below by
\begin{equation}
\label{thm:IhFD:1}
\lambda_{min}\big ( (I_{h}^{\text{FD}})^T I_{h}^{\text{FD}} \big ) \ge C \left (\frac{a_h}{h}\right )^{2d},
\end{equation}
where $C$ is a positive constant.
\item[(iii)] $I_h^{\text{FD}}$ has full column rank.
\end{itemize}
\end{thm}

\begin{proof}
(i) Lemma~\ref{lem:simplex-cube} implies that any element $K$ of $\mathcal{T}_h$ contains
a cube of side length $\frac{2 a_h}{(d+1)\sqrt{d}}$. Thus, when $h_{\text{\text{FD}}}$ satisfies (\ref{hFD-1}), $K$
contains at least a cubic cell of $\mathcal{T}_{\text{FD}}$.
As a consequence, for any vertex (say $\V{x}_j$) of $K$, there is a node (say $\V{x}_k^{\text{FD}}$)
of $\mathcal{T}_{\text{FD}}$ that is in $K$ and its distance to the facet opposing $\V{x}_j$ is at least $h_{\text{\text{FD}}}$.
From Lemma~\ref{lem:simplex-barycentric}, the barycentric coordinate of $\V{x}_k^{\text{FD}}$
at $\V{x}_j$ is greater than or equal to $h_{\text{\text{FD}}}/a_j \ge a_h/(h (d+1) \sqrt{d})$.
Then, (\ref{thm:IhFD:0}) follows from (\ref{lem:IhFD-1:1}) and $D_h$ is invertible.

(ii) For any function $\vec{u} =\{ u(\V{x}_j), j = 1, ..., N_v\}$,
\begin{equation}
\vec{u}^T (I_h^{\text{FD}})^T I_h^{\text{FD}} \vec{u} = \sum_{K \in \mathcal{T}_h} \sum_{\V{x}_k^{\text{FD}} \in K}
\big ( I_h u |_K (\V{x}_k^{\text{FD}}) \big )^2 .
\label{thm:IhFD:2}
\end{equation}
As mentioned in the proof of (i), each element of $\mathcal{T}_h$ contains at least a cubic cell of $\mathcal{T}_{\text{FD}}$.
We can take $\tilde{K}$ in Lemma~\ref{lem:simplex-interp}
as a simplex formed by any $d+1$ vertices of the cubic cell. Then $|\tilde{K}|/|K| \ge C (a_h/h)^d$ for some constant $C$.
From Lemma~\ref{lem:simplex-interp}, we have
\[
\sum_{\V{x}_k^{\text{FD}} \in K}
\big ( I_h u |_K (\V{x}_k^{\text{FD}}) \big )^2 \ge C\left  (\frac{a_h}{h}\right )^{2 d} \sum_{j=0}^d (u_j^K)^2 ,
\]
which yields (\ref{thm:IhFD:1}).

(iii) is a consequence of (ii).
\end{proof}

\begin{rem}
\label{rem:Ah-3}
The choice (\ref{hFD-1}) is needed for the theoretical guarantee of the full column rank of $I_h^{\text{FD}}$ and the invertibility of $D_h$.
However, the requirement is only a sufficient condition. Numerical experiment shows that we can use much larger $h_{\text{\text{FD}}}$, for instance,
$h_{\text{\text{FD}}} = a_h$, which works well for the examples we have tested.
\qed
\end{rem}

\subsection{Linear systems, stability, and convergence}

The GoFD discretization of the homogeneous Dirichlet problem (\ref{FL-1})
on the unstructured mesh $\mathcal{T}_h$ is defined as
\begin{equation}
\label{GoFD-1}
\frac{1}{h_{\text{\text{FD}}}^{2s}} D_h^{-1} (I_{h}^{\text{FD}})^T A_{\text{FD}} I_{h}^{\text{FD}} \vec{u}_h = D_h^{-1} (I_{h}^{\text{FD}})^T\vec{f}_{\text{FD}},
\end{equation}
where
\begin{align*}
&\vec{u}_h = \{ u_{j} \approx u(\V{x}_j),\; j = 1, ..., N_v; u_j = 0, \text{ for } \V{x}_j \in \Omega^c \},
\\
&\vec{f}_{\text{FD}} = \{ f(\V{x}_k^{\text{FD}}),\; k = 1, ..., N_v^{\text{FD}}; f(\V{x}_k^{\text{FD}}) = 0, \text{ for } \V{x}_k^{\text{FD}} \in \Omega^c \}.
\end{align*}
Notice that $u$ is approximated on the vertices of $\mathcal{T}_h$ while the right-hand side function
$f$ is calculated on the vertices of $\mathcal{T}_{\text{FD}}$.
We can also use the values of $f$ at the vertices of $\mathcal{T}_h$.
In this case, we have
\begin{equation}
\label{GoFD-3}
\frac{1}{h_{\text{\text{FD}}}^{2s}} D_h^{-1} (I_{h}^{\text{FD}})^T A_{\text{FD}} I_{h}^{\text{FD}} \vec{u}_h = \vec{f}_{h}.
\end{equation}
Numerical experiment shows that (\ref{GoFD-1}) and (\ref{GoFD-3}) produce comparable results.
Since (\ref{GoFD-1}) provides some convenience in defining the local truncation error (cf. (\ref{GoFD-lte})),
we use (\ref{GoFD-1}) in this work.

The system (\ref{GoFD-1}) can be simplified into
\begin{equation}
\label{GoFD-2}
(I_{h}^{\text{FD}})^T A_{\text{FD}} I_{h}^{\text{FD}}  \vec{u}_h = h_{\text{\text{FD}}}^{2 s} (I_{h}^{\text{FD}})^T \vec{f}_{\text{FD}} .
\end{equation}
One may notice that $D_h^{-1}$ does not appear in the above equation. This is due to the special choice of the right-hand side
function of (\ref{GoFD-1}). It appears in (\ref{GoFD-3}) for a different choice of the right-hand side function.
Moreover, since $I_{h}^{\text{FD}} $ is sparse and the multiplication of $A_{\text{FD}}$
with vectors can be carried out efficiently using FFT (cf. Section~\ref{SEC:AU-FFT}),
(\ref{GoFD-2}) is amenable to iterative solution with Krylov subspace methods.
The conjugate gradient method (CG) is used in our computation.
 Recall that the cost for each iteration of CG is proportional to the cost of computing
the matrix-vector multiplication $(I_{h}^{\text{FD}})^T A_{\text{FD}} I_{h}^{\text{FD}} \vec{u}_h$, which can be estimated as
$\mathcal{O}(N^d \log (N^d)) + \mathcal{O}(N_{v})$, where $(2N+1)$ is the number of grid points of $\mathcal{T}_{\text{FD}}$
in each axial direction and $N_v$ is the number of vertices of $\mathcal{T}_h$. When $\mathcal{T}_h$ is quasi-uniform,
the choice $h_{\text{FD}} = \mathcal{O}(a_h)$ leads to $N^d = \mathcal{O}(N_e) = \mathcal{O}(N_v)$, where $N_e$ denotes the number of
elements of $\mathcal{T}_h$. This gives the cost of each CG iteration for solving (\ref{GoFD-2}) as $\mathcal{O}(N_e \log N_e)$.
When $\mathcal{T}_h$ is not quasi-uniform, it is difficult to estimate the cost since, in this case, $a_h$ can be very different from $h$.

\begin{rem}
It is worth noting that only the block of the system (\ref{GoFD-2}) corresponding to the interior vertices is solved
in the actual computation since the unknown variables on $\partial \Omega$ are known.
\qed
\end{rem}

\begin{thm}[Stability]
\label{thm:GoFD-stab}
If $h_{\text{\text{FD}}}$ satisfies (\ref{hFD-1}), then
the solution of (\ref{GoFD-1}) satisfies
\begin{equation}
\label{thm:GoFD-stab:1}
\vec{u}_{h}^T \vec{u}_{h} \le C \, N_{val}  N_{\text{FD}}^h  \left (\frac{h}{a_h}\right )^{4d} \, \vec{f}_{\text{FD}}^T \vec{f}_{\text{FD}} ,
\end{equation}
 where $h$ and $a_h$ are the maximum diameter and minimum height of elements of $\mathcal{T}_h$.
\end{thm}
\begin{proof}
Multiplying $\vec{u}_{h}^T$ from left with \eqref{GoFD-2} and using the Cauchy-Schwarz inequality and Lemma~\ref{lem:IhFD-1}, we get
\begin{align*}
\vec{u}_h^T (I_{h}^{\text{FD}})^T A_{\text{FD}} I_{h}^{\text{FD}}  \vec{u}_h
& =   h_{\text{\text{FD}}}^{2 s} ( I_{h}^{\text{FD}} \vec{u}_h)^T \vec{f}_{\text{FD}}
\\
& \le h_{\text{\text{FD}}}^{2 s}  ( ( I_{h}^{\text{FD}} \vec{u}_h)^T  I_{h}^{\text{FD}} \vec{u}_h )^{\frac{1}{2}}
( ( \vec{f}_{\text{FD}} )^T \vec{f}_{\text{FD}} )^{\frac{1}{2}}
\\
& \le C\, h_{\text{\text{FD}}}^{2 s}  (N_{val}  N_{\text{FD}}^h)^{\frac 12}
( \vec{u}_h^T  \vec{u}_h )^{\frac{1}{2}} ( ( \vec{f}_{\text{FD}} )^T \vec{f}_{\text{FD}} )^{\frac{1}{2}} .
\end{align*}
Moreover, from Proposition~\ref{pro:AFD} and Theorem~\ref{thm:IhFD} we have
\begin{align*}
\vec{u}_h^T (I_{h}^{\text{FD}})^T A_{\text{FD}} I_{h}^{\text{FD}}  \vec{u}_h
& \ge C h_{\text{\text{FD}}}^{2 s}  \vec{u}_h^T (I_{h}^{\text{FD}})^T I_{h}^{\text{FD}}  \vec{u}_h
 \ge C h_{\text{\text{FD}}}^{2 s}  \left (\frac{a_h}{h}\right )^{2d}  \vec{u}_h^T  \vec{u}_h .
\end{align*}
Combining the above results, we get
\[
\vec{u}_h^T  \vec{u}_h \le C\, (N_{val}  N_{\text{FD}}^h)^{\frac 12}  \left (\frac{h}{a_h}\right )^{2d}
( \vec{u}_h^T  \vec{u}_h )^{\frac{1}{2}} ( ( \vec{f}_{\text{FD}} )^T \vec{f}_{\text{FD}} )^{\frac{1}{2}} ,
\]
which gives rise to (\ref{thm:GoFD-stab:1}).
\end{proof}

Denote the exact solution of (\ref{FL-1}) by $u = u^e(\V{x})$. We define the local truncation error as
\begin{equation}
\label{GoFD-lte}
\vec{\tau}_{\text{FD}} = \vec{f}_{\text{FD}} - \frac{1}{h_{\text{\text{FD}}}^{2 s}} A_{\text{FD}} I_h^{\text{FD}} \vec{u}_h^{\,e} ,
\end{equation}
where
\[
\vec{u}_h^{\, e} =  \{ u^e(\V{x}_j),\; j = 1, ..., N_v; u^e(\V{x}_j) = 0, \text{ for } \V{x}_j \in \Omega^c \} .
\]
$\vec{\tau}_{\text{FD}}$ can be rewritten into
\begin{equation}
\label{GoFD-lte-2}
\vec{\tau}_{\text{FD}} = \vec{f}_{\text{FD}} - \frac{1}{h_{\text{\text{FD}}}^{2 s}} A_{\text{FD}} \vec{u}_{\text{FD}}^{\,e}
+ \frac{1}{h_{\text{\text{FD}}}^{2 s}} A_{\text{FD}} \big (\vec{u}_{\text{FD}}^{\,e} - I_h^{\text{FD}} \vec{u}_h^{\,e} \big ).
\end{equation}
Thus, $\vec{\tau}_{\text{FD}}$ can be viewed as a combination of the discretization error on the uniform grid $\mathcal{T}_{\text{FD}}$
and the interpolation error from $\mathcal{T}_h$ to $\mathcal{T}_{\text{FD}}$.
From (\ref{GoFD-lte}), we have
\[
(I_h^{\text{FD}})^T A_{\text{FD}} I_h^{\text{FD}} \vec{u}_h^{\,e} = h_{\text{\text{FD}}}^{2 s} (I_h^{\text{FD}})^T \vec{f}_{\text{FD}} - h_{\text{\text{FD}}}^{2 s} (I_h^{\text{FD}})^T \vec{\tau}_{\text{FD}} .
\]
Subtracting (\ref{GoFD-2}) from the above equation, we obtain the error equation as
\begin{equation}
\label{GoFD-error-1}
 (I_h^{\text{FD}})^T A_{\text{FD}} I_h^{\text{FD}} \vec{e}_h = - h_{\text{\text{FD}}}^{2 s} (I_h^{\text{FD}})^T \vec{\tau}_{\text{FD}} ,
 \end{equation}
 where the error is defined as $\vec{e}_h = \vec{u}_h - \vec{u}_h^{\,e}$. From Theorem~\ref{thm:GoFD-stab}, we have the following
 corollary.

\begin{co}[Convergence]
\label{co:GoFD-error}
If $h_{\text{\text{FD}}}$ satisfies (\ref{hFD-1}), the error for the GoFD scheme (\ref{GoFD-1}) is bounded by
\begin{equation}
\label{co:GoFD-error:1}
\vec{e}_h^T \vec{e}_h
\le C N_{val} N_{\text{FD}}^h \left (\frac{h}{a_h}\right )^{4 d}
\vec{\tau}_{\text{FD}}^T \vec{\tau}_{\text{FD}} .
\end{equation}
\end{co}

\begin{rem}
Here we do not attempt to give a rigorous analysis of the local truncation error since it is still challenging to
do so for the uniform FD discretization for solutions of optimal regularity (see Remark~\ref{rem:FD-error}). Instead,
we provide some intuitions here. From (\ref{GoFD-lte-2}) we see that the local truncation error
consists of two parts, one from the uniform FD discretization and the other from linear interpolation.
It is known \cite[Proposition 1.2]{Borthagaray-2018} that the linear interpolation error in $L^2$ norm
is $\mathcal{O}(h^{\min(1,s+1/2-\epsilon)})$
for functions in $H^{s+1/2-\epsilon}(\Omega)$ for any $\epsilon > 0$.
Moreover, it can be proved that $h_{\text{\text{FD}}}^{-2 s}  A_{\text{FD}}$ is bounded in $H^{s}(\Omega)$.
Thus, we can expect that the local truncation error (and thus the error by Corollary~\ref{co:GoFD-error})
for the GoFD scheme (\ref{GoFD-1}) is $\mathcal{O}(h^{\min(1,s+1/2-\epsilon)})$ in $L^2$ norm
if the local truncation error of the uniform FD discretization is in the same order (cf. Remark~\ref{rem:FD-error}).
\qed
\end{rem}

\begin{rem}
\label{rem:FEM-rate}
Interestingly, Borthagaray et al. \cite{Borthagaray-2018} and Acosta et al. \cite{Acosta201701}
show that the error of the linear finite element approximation
of (\ref{FL-1}) in $L^2$ norm is $\mathcal{O}(h^{\min(1,s+1/2)-\epsilon})$ for quasi-uniform meshes and
$\mathcal{O}(\bar{h}^{1 + s})$ for graded meshes.
Here, $\bar{h} \equiv N_e^{-\frac{1}{d}}$ is the average element diameter commonly used
to measure convergence order in mesh adaptation.
The convergence order, $\mathcal{O}(\bar{h}^{1 + s})$,
has also been established by Ainsworth and Glusa \cite{Ainsworth-2017}
for adaptive finite element approximations. Numerical results in Section~\ref{SEC:numerics} show that GoFD has
similar convergence behavior  for quasi-uniform meshes and second-order convergence (in $L^2$ norm) for adaptive meshes.
\qed
\end{rem}

\subsection{Preconditioning with sparse matrices}
\label{SEC:preconditioner}

Various types of preconditioners have been developed for $A_{\text{FD}}$, including
circulant preconditioners \cite{Chan-1996} and the direct use of the Laplacian \cite{Ying-2020}.
In principle, we can use these preconditioners to replace $A_{\text{FD}}$ in the stiffness matrix $(I_{h}^{\text{FD}})^T A_{\text{FD}} I_{h}^{\text{FD}}$
and obtain a preconditioner for (\ref{GoFD-2}).
Here we consider preconditioners based on sparse matrices.
Notice that the fractional Laplacian approaches to the Laplacian operator as $s \to 1$ and the identity operator
as $s \to 0$. Thus, it is reasonable to build an efficient preconditioner
based on the Laplacian at least when $s$ is close to 1.
First,  we choose a sparsity pattern based on the FD discretization
of the Laplacian. For example, we can take the 5-point pattern (cf. (\ref{5-point})) or the 9-point pattern.
Then, we form a sparse matrix using the entries of $A_{\text{FD}}$ at the positions specified by the pattern.
We denote these matrices by $A_{\text{FD}}^{(5)}$ and $A_{\text{FD}}^{(9)}$, respectively.
Next, we define
\begin{equation}
\label{Ah-4}
A_h^{(5)} =  (I_{h}^{\text{FD}} )^T A_{\text{FD}}^{(5)} I_{h}^{\text{FD}} ,
\qquad A_h^{(9)} = (I_{h}^{\text{FD}} )^T A_{\text{FD}}^{(9)} I_{h}^{\text{FD}}  .
\end{equation}
Finally, the preconditioners for (\ref{GoFD-2}) are obtained using the incomplete Cholesky decomposition
of $A_h^{(5)}$ and $A_h^{(9)}$ with level-1 fill-ins. Notice that all of $A_{\text{FD}}^{(5)}$ and $A_{\text{FD}}^{(9)}$ and therefore,
$A_h^{(5)}$ and $A_h^{(9)}$ are sparse and they can be computed economically.
Effectiveness of these preconditioners will be demonstrated in numerical examples.

\section{Mesh adaptation}
\label{SEC:MMPDE}

It is known (e.g. see \cite{Borthagaray-2018,Ros-Oton-2014})
that the solution of (\ref{FL-1}) has low regularity especially near the boundary of $\Omega$.
Thus, it is useful to use mesh adaptation in the numerical solution of (\ref{FL-1}) to improve accuracy and convergence order.
We recall that GoFD described in the previous section uses unstructured meshes for $\Omega$,
which not only works for arbitrary geometry of $\Omega$ but also allows easy incorporation with
existing mesh adaptation algorithms.

\begin{algorithm}[htb]
\caption{Adaptive mesh grid-overlay finite difference method}
\label{alg:amGoFD}
\begin{itemize}
\item[-] Given an initial mesh $\mathcal{T}_h^{(0)}$ for $\Omega$.
\item[-] For $\ell = 1, \ldots, \ell_{max}$
	\begin{itemize}
	\item[-] Solve (\ref{GoFD-2}) on $\mathcal{T}_h^{(\ell)}$ for $u_h^{(\ell)}$.
	\item[-] Generate a new mesh $\mathcal{T}_h^{(\ell+1)}$ using the MMPDE method based on
		    $u_h^{(\ell)}$ and $\mathcal{T}_h^{(\ell)}$.
	\end{itemize}

\item[-] end $\ell$	
\end{itemize}
\end{algorithm}

We use here the MMPDE moving mesh method for mesh adaptation.
The procedure for combining GoFD with the MMPDE method is given in Algorithm~\ref{alg:amGoFD}.
We use $\ell_{max} = 5$ in our computation. Numerical experiment shows that this is sufficient.

The MMPDE method is used to generate the new mesh $\mathcal{T}_h^{(\ell+1)}$
for $\Omega$. The method has been
developed (e.g., see \cite{HK2015,HRR94a,HR11}) for general purpose of mesh adaptation and movement.
It uses the moving mesh PDE (or moving mesh equations in discrete form) to move vertices
continuously in time and in an orderly manner in space.
A key idea of the MMPDE method is viewing any nonuniform mesh as a uniform one in some Riemannian metric
specified by a tensor $\M = \M(\V{x})$.  The metric tensor provides the information needed to control the size, shape, and
orientation of mesh elements throughout the domain. Various metric tensors have been developed in \cite{HS03}.
For the current work, we employ a Hessian-based metric tensor
\begin{equation}
\label{MK-1}
\M_{K} =\det \left(I+\frac{1}{\alpha_h}|H_K(u_h^{(\ell)})|
\right)^{-\frac{1}{d+4}} \left(I+\frac{1}{\alpha_h}|H_K(u_h^{(\ell)})|\right) ,\quad \forall K \in \mathcal{T}_h^{(\ell)}
\end{equation}
where $\det (\cdot)$ denotes the determinant of a matrix,
$H_K(u_h^{(\ell)})$ is a recovered Hessian of $u_h^{(\ell)}$ on the element $K$ (through quadratic least squares fitting),
$|H_K(u_h^{(\ell)})| = \sqrt{H_K(u_h^{(\ell)})^2}$,
and $\alpha_h$ is a regularization parameter defined through the following algebraic equation:
\[
\sum_{K}|K|\, \det\left(I+\frac{1}{\alpha_h}|H_K(H_K(u_h^{(\ell)}))|\right)^{\frac{2}{d+4}}
 =2 |\Omega|.
\]
This metric tensor is known to be optimal for the $L^2$-norm of linear interpolation error \cite{HS03}.

It is known (e.g., see \cite{HK2015,HR11}) that a uniform simplicial mesh $\mathcal{T}_h$ in metric $\M$
satisfies the following equidistribution and alignment conditions,
\begin{align}
  \sqrt{\det(\M_K)}\; |K| &=\frac{\sigma_h}{N_e},\;\qquad \forall K\in \mathcal{T}_h
\label{C-1}
\\
  \frac{1}{2}\text{trace}\left((F'_{K})^{-1} \M_K^{-1} (F'_{K})^{-T}\right)&
  =\det\left((F'_{K})^{-1} \M_K^{-1} (F'_{K})^{-T}\right)^{\frac{1}{2}},\;\qquad \forall K\in \mathcal{T}_h
\label{C-2}
\end{align}
where  $N_e$ denotes the number of elements in $\mathcal{T}_h$,
$F'_K$ is the Jacobian matrix of the affine mapping
$F_K: \hat{K}\to K$, $\hat{K}$ is the reference element taken as an equilateral simplex with unit volume, and
\[
\sigma_h =\sum_{K} \sqrt{\det(\M_K)}\; |K|.
\]
The condition \eqref{C-1} requires all elements to have the same size while \eqref{C-2} requires every element $K$
to be similar to $\hat{K}$, in metric $\M_K$.
An energy function associated with these conditions is given by
\begin{align}
I_h=&\frac{1}{3}\sum_{K}\sqrt{\det(\M_K)}\; |K| \text{trace}\left((F'_{K})^{-1}\M_K^{-1} (F'_{K})^{-T}\right)^{\frac{3 d}{4}}
\notag \\
& \qquad \qquad + \frac{d^{\frac{3d }{4}}}{3}\sum_{K}\sqrt{\det(\M_K)}\; |K|\left(\sqrt{\det(\M_K)} \det(F'_K)\right)^{-\frac{3 d}{4}}.
\label{I-h}
\end{align}
This function is a Riemann sum of a continuous functional developed based on mesh equidistribution and alignment (e.g., see \cite{HR11}).

The energy function $I_h$ is a function of the coordinates of the vertices of $\mathcal{T}_h$, i.e.,
$I_h = I_h(\V{x}_1, ..., \V{x}_{N_v})$.
An approach for minimizing this function is to integrate the gradient system of $I_h$. Thus, we define the moving mesh equations as
\begin{equation}\label{x-1}
   \frac{d \V{x}_i}{d t}=-\frac{\sqrt{\det(\M(\V{x}_i))}}{\tau}\frac{\partial I_h}{\partial \V{x}_i},\qquad i=1,\ldots,N_v
 \end{equation}
where $\tau>0$ is a parameter used to adjust the time scale of mesh movement.
The analytical expression of the derivative of $I_h$ with respect to $\V{x}_i$ can be found using scalar-by-matrix
differentiation \cite{HK2015}. Using this expression, we can rewrite (\ref{x-1}) as
\begin{equation}
 \frac{d \V{x}_i}{d t} = \frac{\sqrt{\det(\M(\V{x}_i))}}{\tau} \sum\limits_{K \in \omega_i} |K| \V{v}_{i_K}^K ,\qquad
 i = 1, ..., N_v
\label{mmpde-1}
\end{equation}
where $\V{v}_{i_K}^K$ is the local mesh velocity
contributed by element $K$ to the vertex $\V{x}_i$. The interested reader is referred to  \cite[Equations (38), (40), and (41)]{HK2015} 
for the analytical expression of $\V{v}_{i_K}^K$.

The nodal velocity needs to be modified at boundary vertices. For fixed boundary
vertices, $\frac{d \V{x}_i}{d t}$ should be set to be zero. If $\V{x}_i$ is allowed to slide along the boundary,
the component of $\frac{d \V{x}_i}{d t}$ in the normal direction of the boundary should be set to be zero.

In our computation, the Matlab ODE solver \textit{ode15s} (a variable-step, variable-order solver
based on the numerical differentiation formulas of orders 1 to 5) is used to integrate (\ref{mmpde-1}), with
the Jacobian matrix approximated by finite differences,
over $t \in (0,1]$ with $\tau = 10^{-2}$ and the initial mesh $\mathcal{T}_h^{(\ell)}$.
The obtained mesh is $\mathcal{T}_h^{(\ell+1)}$. Notice that the mesh connectivity is kept fixed during the time integration.
Thus, $\mathcal{T}_h^{(\ell+1)}$ has the same connectivity as $\mathcal{T}_h^{(\ell)}$.

\section{Numerical examples}
\label{SEC:numerics}

In this section we present numerical results obtained with GoFD described in the previous sections
for one 1D, three 2D, and one 3D examples. Three of those examples come from problem (\ref{FL-1}) with the following setting
in different dimensions,

\begin{equation}
\label{exam-0}
\Omega = B(0,1), \quad f =
\frac{2^{2 s} \Gamma(1+s + k) \Gamma(\frac{d}{2}+s + k)} {k! \; \Gamma(\frac{d}{2}+k)} \cdot
P_{k}^{s,\frac{d}{2}-1}\big(2 |\V{x}|^2-1\big ) ,
\end{equation}
where $P_{k}^{s,\frac{d}{2}-1}(\cdot)$ is the Jacobi polynomial of degree $k$ with parameters $(s,\frac{d}{2}-1)$
and $B(0,1)$ is a unit ball centered at the origin. Notice that $f$ is constant for $k = 0$.
This problem has an analytical exact solution
\begin{equation}
\label{exam-2}
u = (1-|\V{x}|^2)^s_+ P_{k}^{s,\frac{d}{2}-1}\big(2 |\V{x}|^2-1\big ).
\end{equation}

In this section, the solution error is plotted against $N$, the number of elements in $\mathcal{T}_h$. The convergence order is measured
in terms of $\bar{h} \equiv N_e^{-1/d}$, the average element diameter for both fixed and adaptive meshes. For a fixed (and almost
uniform) mesh, $\bar{h}$ is equivalent to $h$, the maximum element diameter while for an adaptive mesh, $\bar{h}$ makes more sense
since the elements can have very different diameters. Moreover, we take $R$
(half of the size of the overlay cube) as 1.1 times of half of the diameter of $\Omega$. We have tried 1.0 and 1.2 times and found
no significant difference in the computed solution.
 Furthermore, we take $h_{\text{FD}} = a_h$. This is larger than what is given in the condition (\ref{hFD-1}) but works well
for all examples we have tested. This relation also implies that $\mathcal{T}_h$ will become finer when $\mathcal{T}_h$ is refined.
Particularly, $h_{\text{FD}}$ can become very small for a highly adaptive mesh $\mathcal{T}_h$ with a small element height $a_h$.

\begin{exam}
\label{ex:5}
The first example is the 1D version of problem (\ref{exam-0}). For this problem, the FD scheme described in Section~\ref{SEC:uniformFD}
can be used for uniform meshes but not for adaptive ones.

We consider the case with $k = 0$. The solution error in $L^\infty$ and $L^2$ norm
is plotted in Fig.~\ref{fig:ex5-2} for fixed and adaptive meshes.
For fixed (uniform) meshes, the error behaves like $\mathcal{O}(h^{s})$ in $L^\infty$ norm
and $\mathcal{O}(h^{\min(1,0.5+s)})$ in $L^2$ norm. This is consistent with the observations made
by other researchers; cf. Remark~\ref{rem:FD-error}. The solution error is also shown for adaptive meshes.
Mesh adaptation improves accuracy and convergence order significantly. Indeed, the error decreases like
$\mathcal{O}(\bar{h}^{0.5+s})$ in $L^\infty$ norm and $\mathcal{O}(\bar{h}^{2})$ in $L^2$ norm for adaptive meshes.

Results for $k > 0$ show similar behavior. They are not included here to save space.
\qed
\end{exam}

\begin{figure}[ht]
\centering
\subfigure[$s = 0.25$, with FM]{
\includegraphics[width=0.22\textwidth]{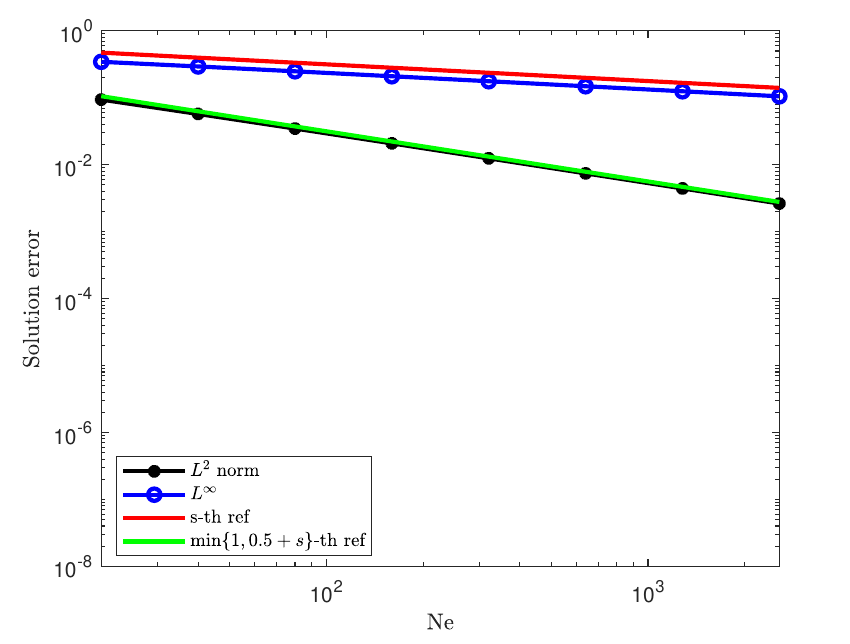}
}
\centering
\subfigure[$s = 0.5$, with FM]{
\includegraphics[width=0.22\textwidth]{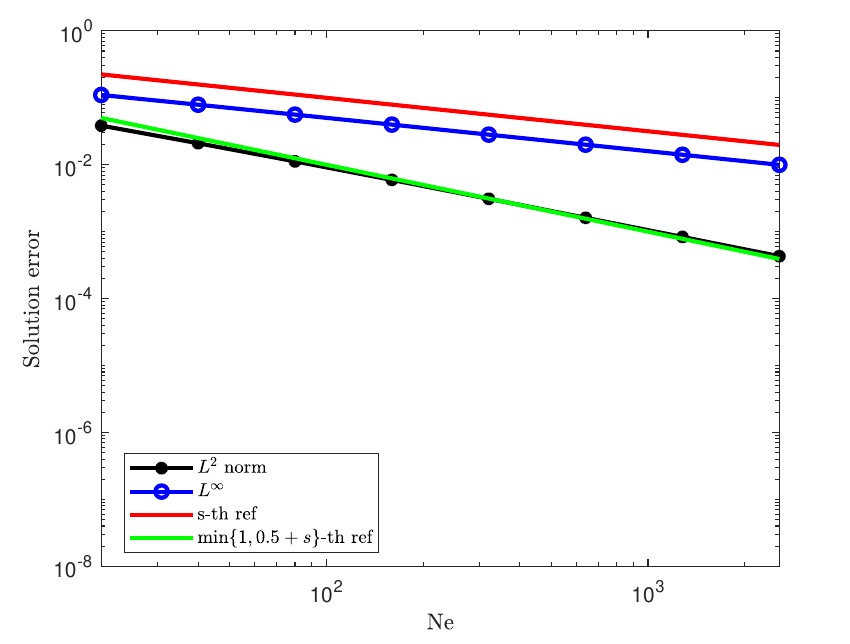}
}
\centering
\subfigure[$s = 0.75$, with FM]{
\includegraphics[width=0.22\textwidth]{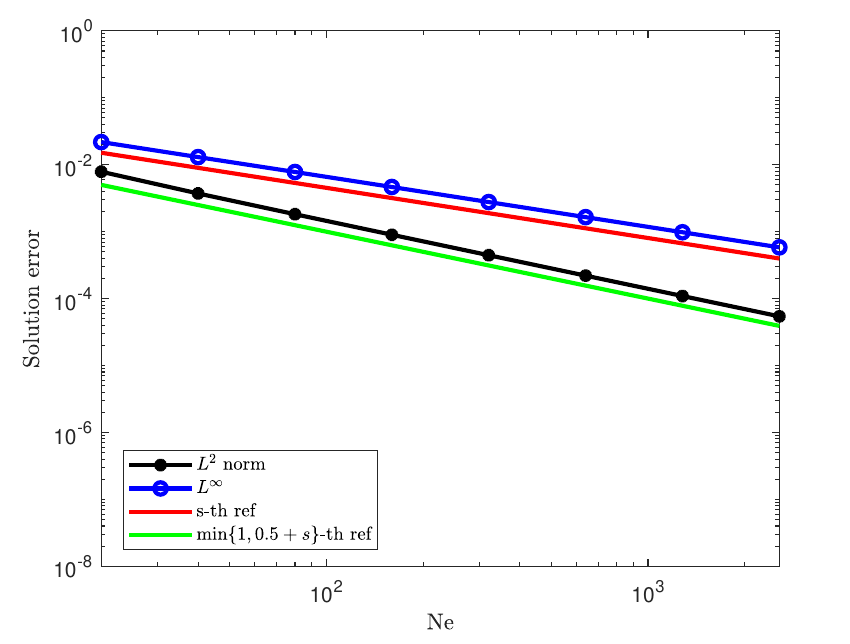}
}
\centering
\subfigure[$s = 0.95$, with FM]{
\includegraphics[width=0.22\textwidth]{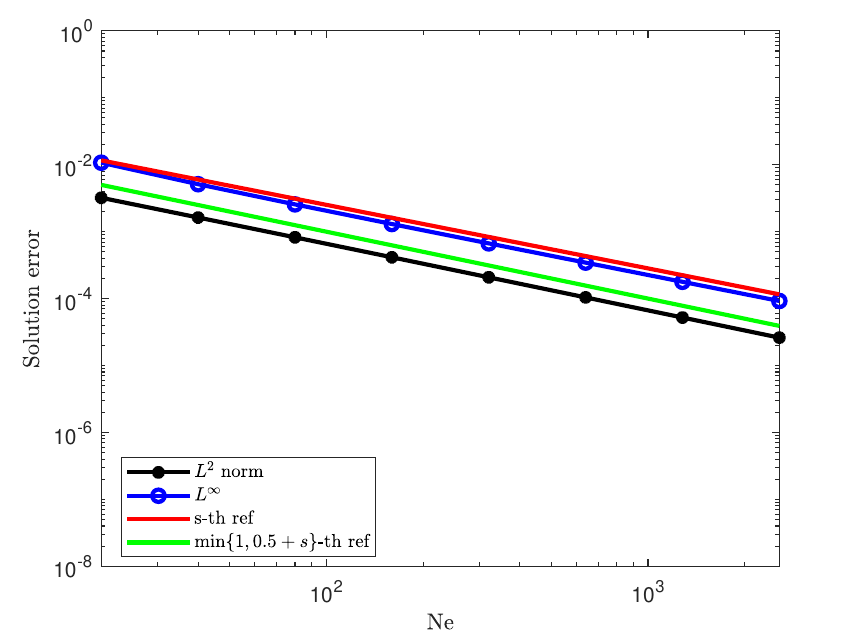}
}
\\
\subfigure[$s = 0.25$, with AM]{
\includegraphics[width=0.22\textwidth]{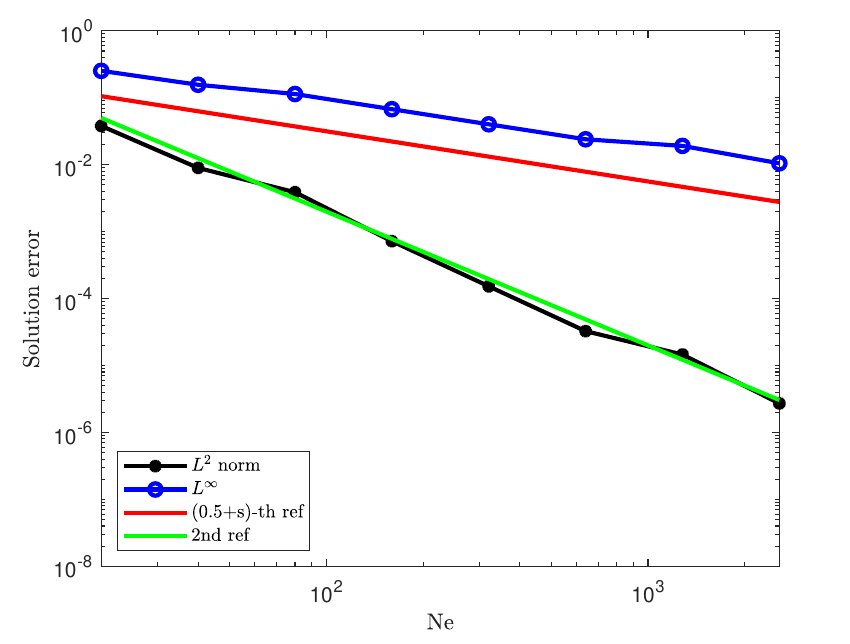}
}
\centering
\subfigure[$s = 0.5$, with AM]{
\includegraphics[width=0.22\textwidth]{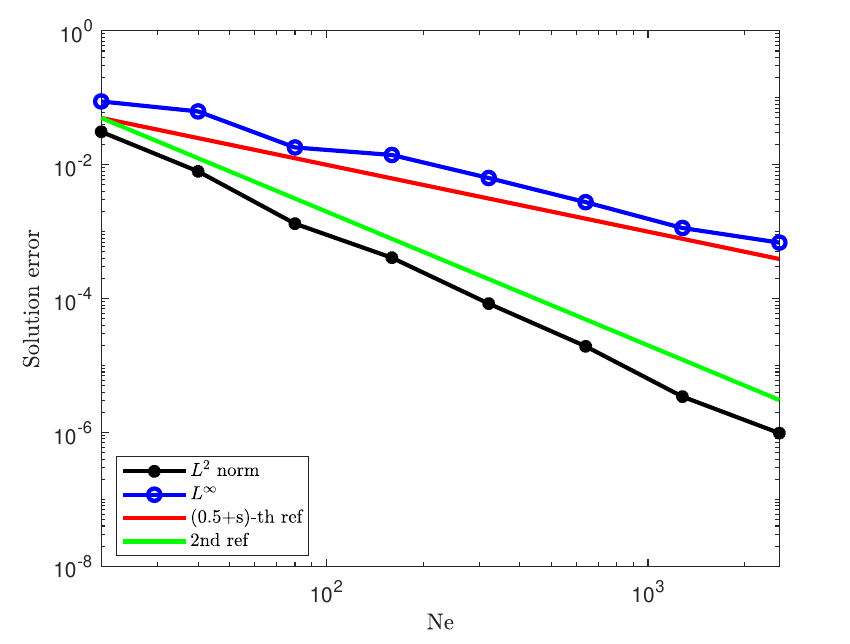}
}
\centering
\subfigure[$s = 0.75$, with AM]{
\includegraphics[width=0.22\textwidth]{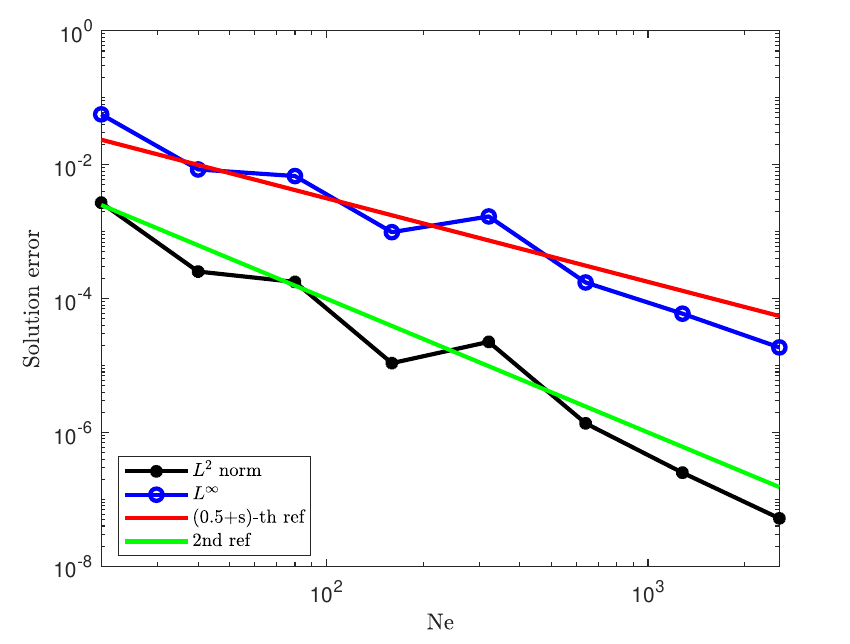}
}
\centering
\subfigure[$s = 0.95$, with AM]{
\includegraphics[width=0.22\textwidth]{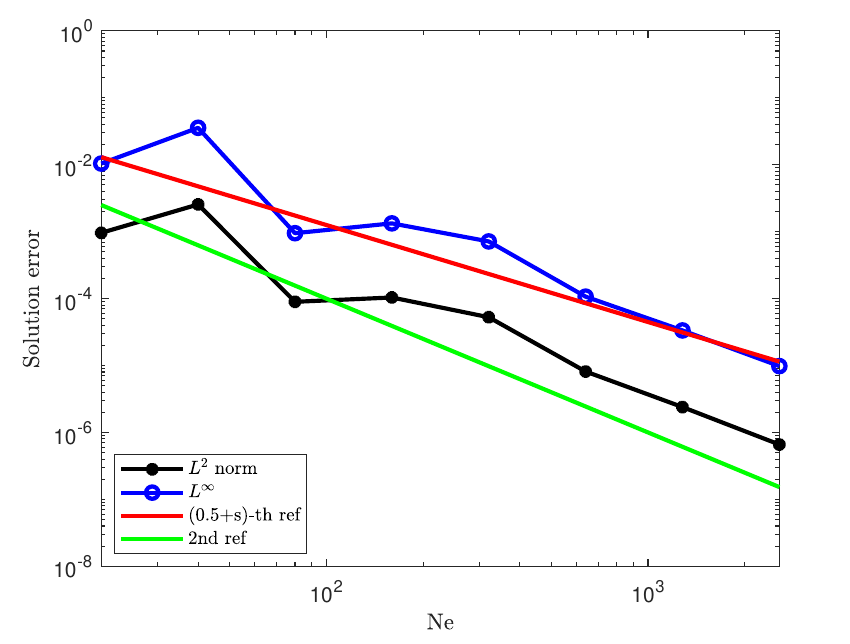}
}
\caption{Example~\ref{ex:5}. The solution error is plotted as a function of $N_e$ (the number of elements of $\mathcal{T}_h$)
for $k = 0$. FM stands for Fixed Mesh and AM stands for Adaptive Mesh.}
\label{fig:ex5-2}
\end{figure}

\begin{exam}\label{ex:1}
The second example is the 2D version of (\ref{exam-0}).
We consider two cases with $k = 0$ and $5$ and $s = 0.5$.
Fig.~\ref{fig:ex1-1} shows computed solutions.
The convergence histories are shown in Fig.~\ref{fig:ex1-2}.
The $L^2$ norm of the solution error converges like $\mathcal{O}(h)$ for fixed meshes.
This is consistent with finite element approximations (cf. Remark~\ref{rem:FEM-rate}) since in this case with $s = 0.5$,
$\mathcal{O}(h^{\min(1,0.5+s)}) = \mathcal{O}(h)$.
On the other hand, the error is second oder, i.e., $\mathcal{O}(\bar{h}^2)$,
for adaptive meshes. This is higher than the expected rate
$\mathcal{O}(\bar{h}^{1+s})=\mathcal{O}(\bar{h}^{1.5})$ (cf. Remark~\ref{rem:FEM-rate}). Higher accuracy with mesh adaptation
can also be observed from the computed solutions. For instance, oscillations are visible in Fig.~\ref{fig:ex1-1}(c)
but not in Fig.~\ref{fig:ex1-1}(d). Examples of adaptive mesh are shown in Fig.~\ref{fig:ex1-0}.

We now examine the effectiveness of the preconditioner described in Section~\ref{SEC:preconditioner}. The convergence history
for the conjugate gradient method (CG) with/without preconditioning is shown in Fig.~\ref{fig:ex1-3}. We can see that the preconditioner
reduces the number of iterations significantly. Moreover, the preconditioner is more effective when $s$ is closer to 1.
Meanwhile, a smaller number of iterations is required to reach the same accuracy for $s = 0.5$ than $s = 0.9$.
These observations are consistent with the fact that the fractional Laplacian approaches
to the Laplacian as $s \to 1$ and the identity operator as $s \to 0$. As a result,
the stiffness matrix of the FD approximation has a smaller condition number and the corresponding
linear system is easier to solve for smaller $s$.
Moreover, the preconditioner, whose pattern is based on that of the FD discretion of the Laplacian, can be expected to be more
effective when the fractional Laplacian is closer to the Laplacian.

The CG convergence history is also plotted in Fig.~\ref{fig:ex1-4} for the preconditioner based on the 5-point pattern.
This preconditioner is slightly less effective than that based on the 9-point pattern.
\qed
\end{exam}

\begin{figure}[ht!]
\centering
\subfigure[$k=0$, with FM]{
\includegraphics[width=0.22\textwidth]{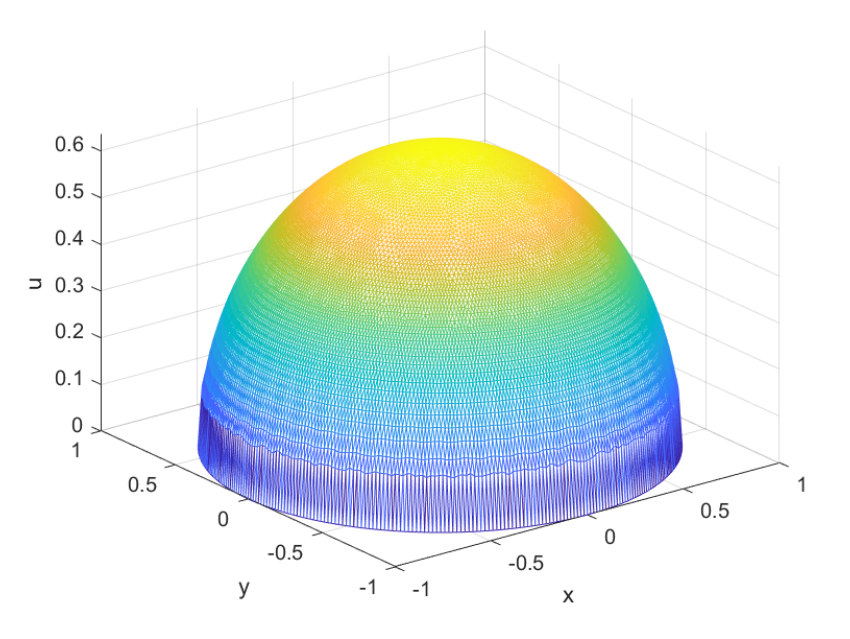}
}
\subfigure[$k=0$, with AM]{
\includegraphics[width=0.22\textwidth]{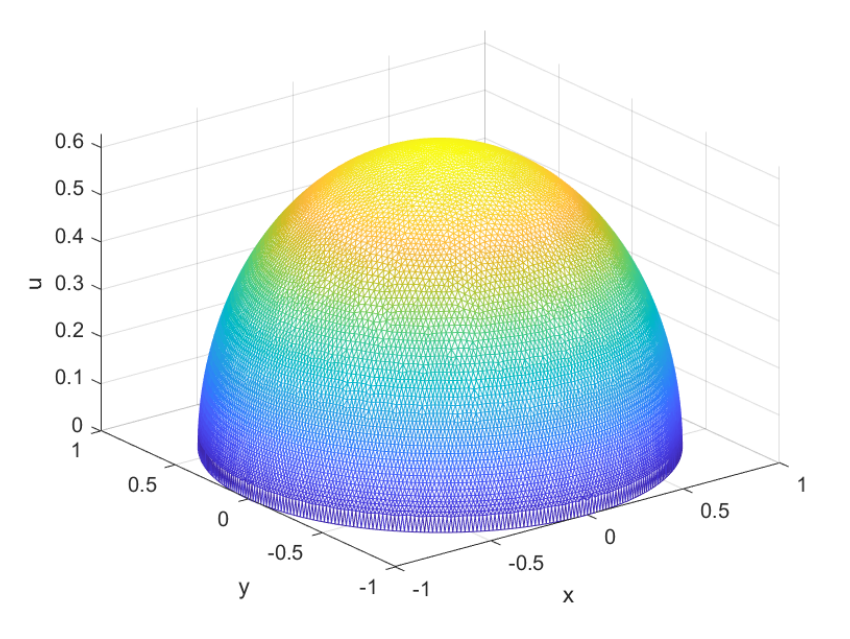}
}
\subfigure[$k=5$, with FM]{
\includegraphics[width=0.22\textwidth]{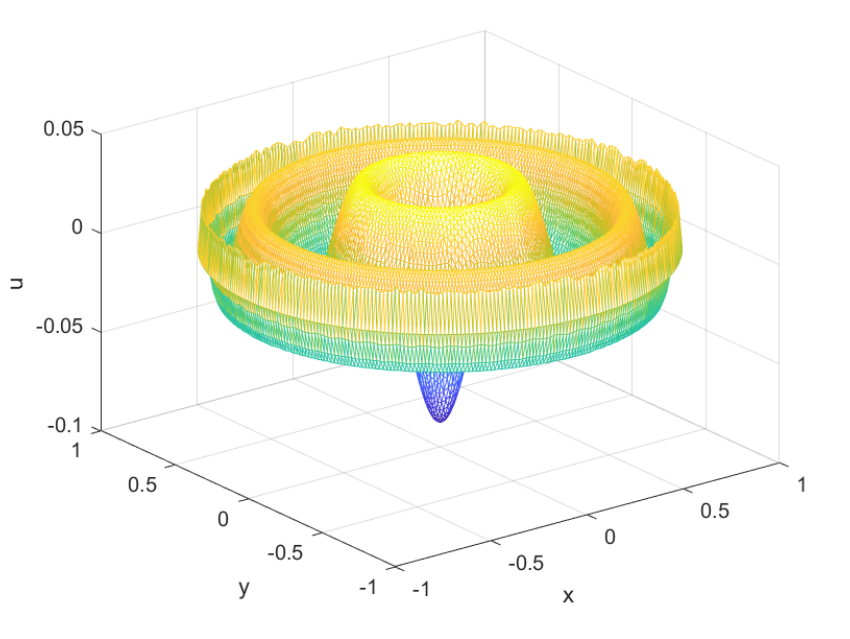}
}
\subfigure[$k=5$, with AM]{
\includegraphics[width=0.22\textwidth]{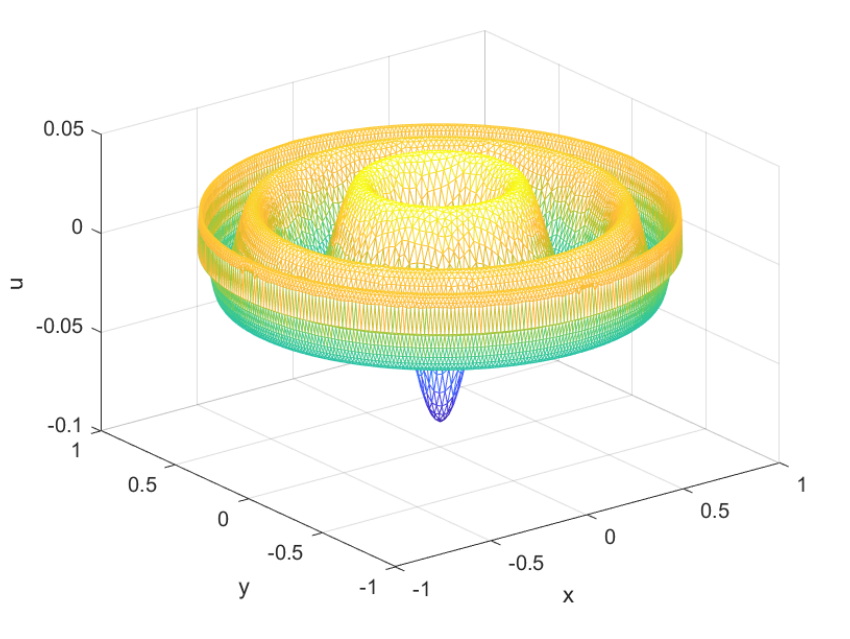}
}
\caption{Example~\ref{ex:1}. Computed solutions obtained with meshes of  $N_e=27130$ for $s=0.5$. FM: fixed mesh and AM: adaptive mesh.}
\label{fig:ex1-1}
\end{figure}

\begin{figure}[ht!]
\centering
\subfigure[$k=0$]{
\includegraphics[width=0.25\textwidth]{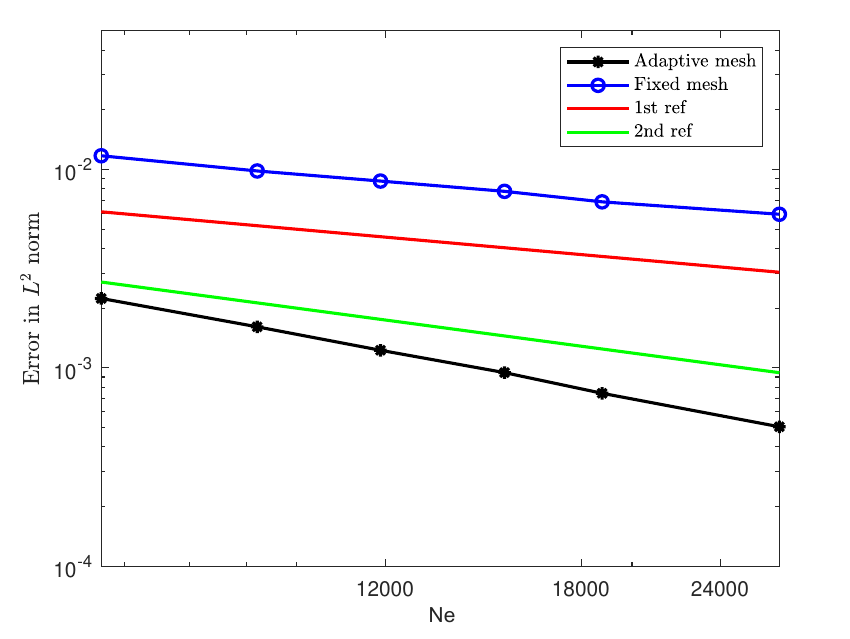}
}
\quad
\subfigure[$k=5$]{
\includegraphics[width=0.25\textwidth]{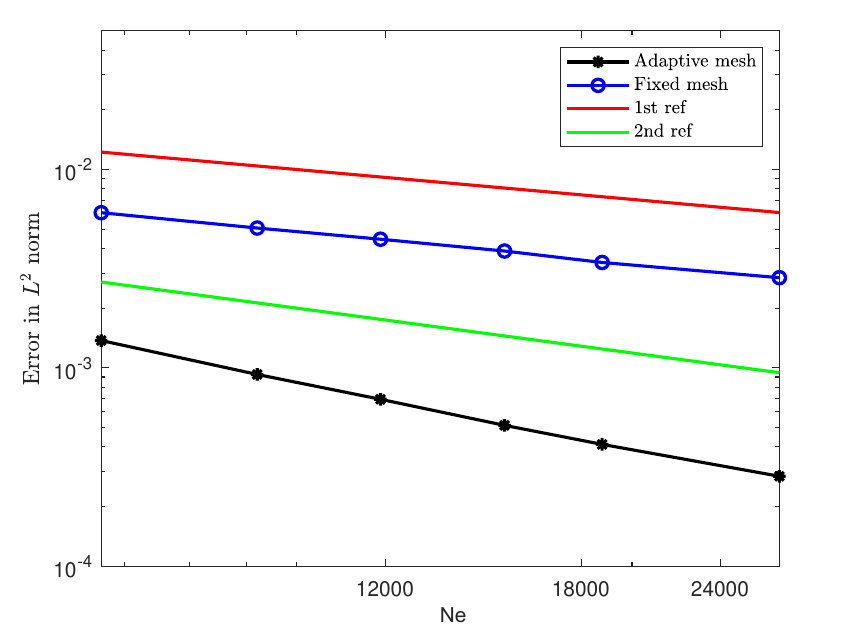}
}
\caption{Example~\ref{ex:1}. The $L^2$ norm of the solution error is plotted as a function of $N_e$ for $k=0$ and 5 and
$s=0.5$ with and without mesh adaptation.
}
\label{fig:ex1-2}
\end{figure}

\begin{figure}[ht!]
\centering
\subfigure[$k=0$]{
\includegraphics[width=0.25\textwidth]{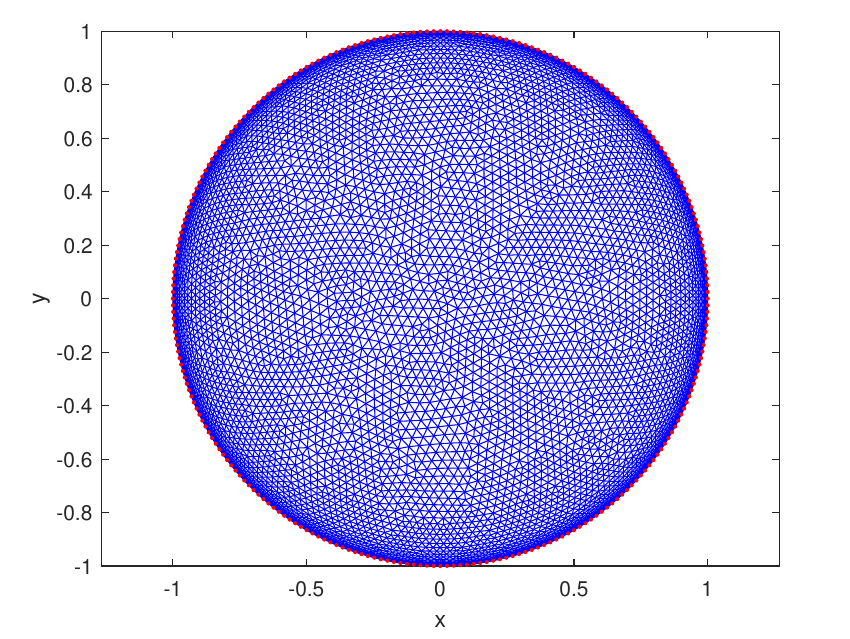}
}
\quad
\subfigure[$k=5$]{
\includegraphics[width=0.25\textwidth]{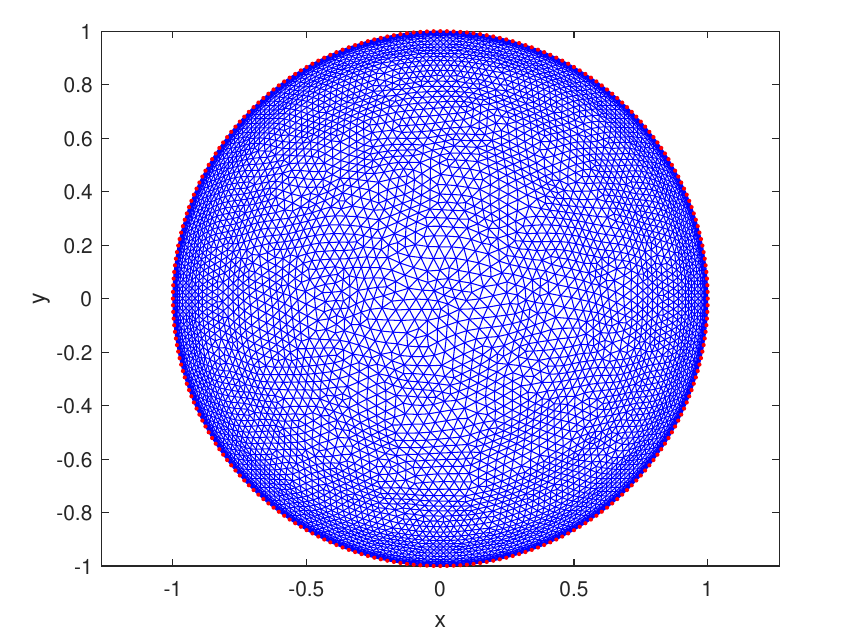}
}
\caption{Example~\ref{ex:1}. Adaptive meshes of $N_e=11886$ for $s=0.5$.}
\label{fig:ex1-0}
\end{figure}

\begin{figure}[ht!]
\centering
\subfigure[$k=0$, $s = 0.5$]{
\includegraphics[width=0.22\textwidth]{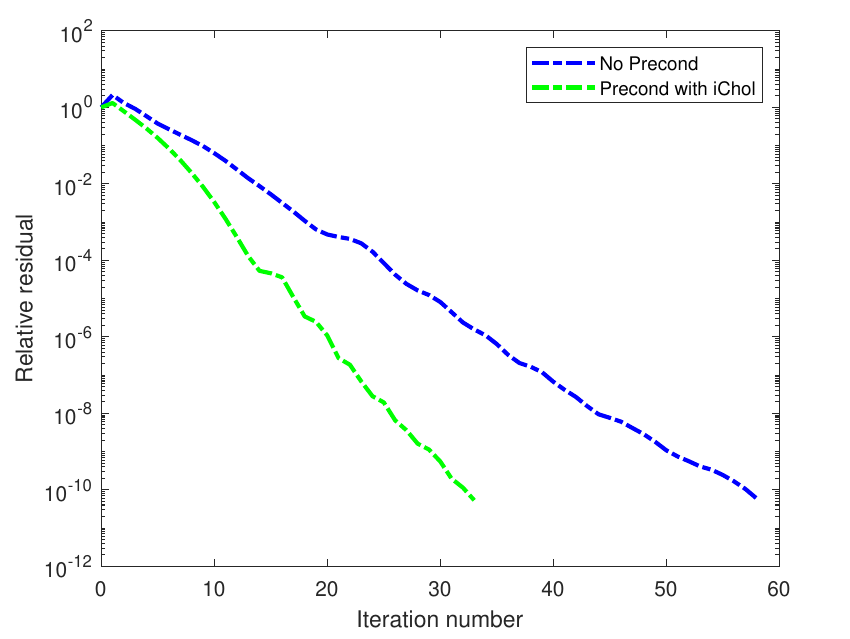}
}
\subfigure[$k=5$, $s = 0.5$]{
\includegraphics[width=0.22\textwidth]{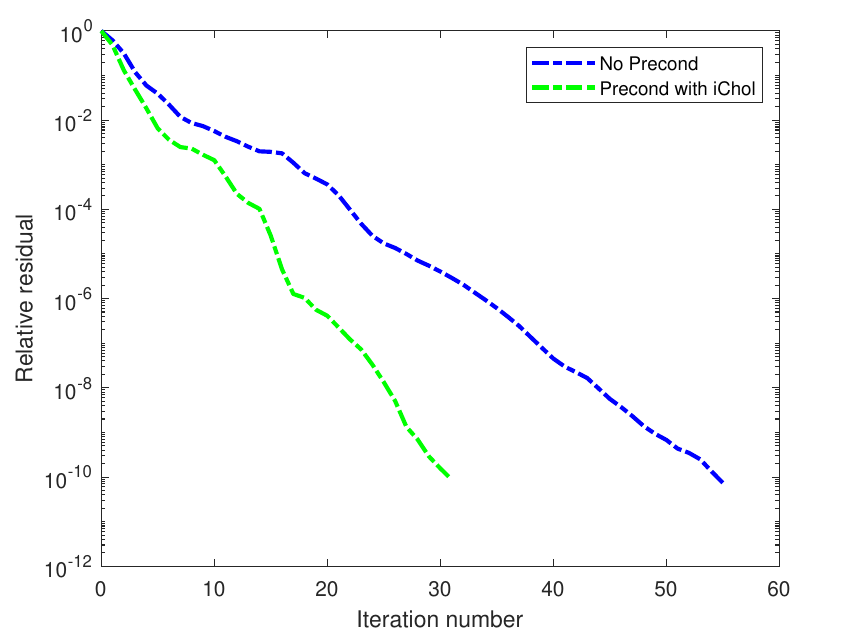}
}
\subfigure[$k=0$, $s = 0.9$]{
\includegraphics[width=0.22\textwidth]{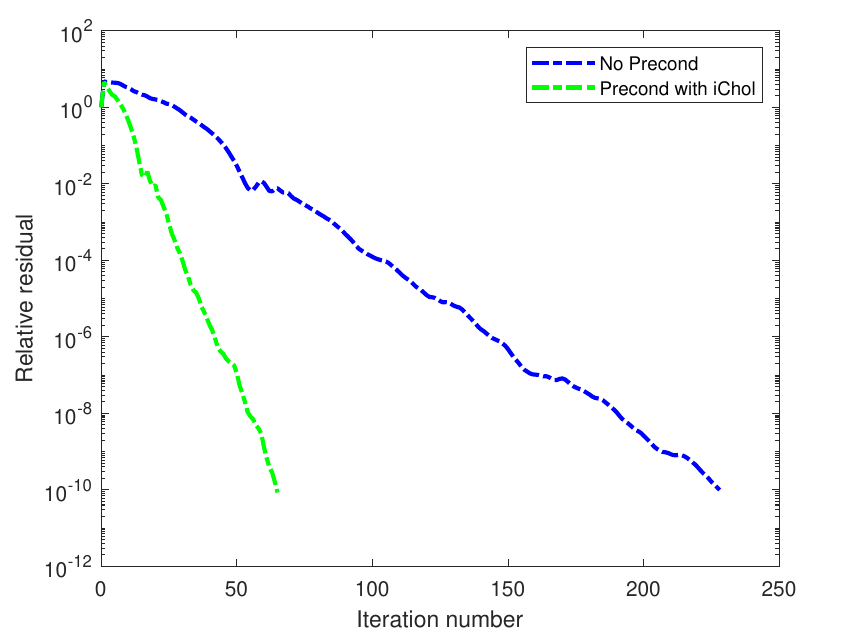}
}
\subfigure[$k=5$, $s = 0.9$]{
\includegraphics[width=0.22\textwidth]{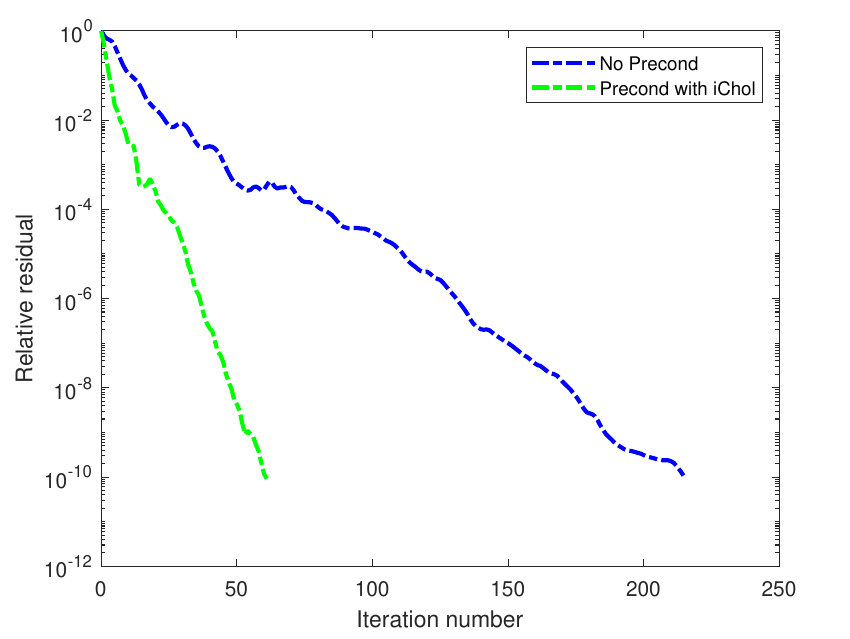}
}
\caption{Example~\ref{ex:1}. The CG convergence history is plotted for a fixed
mesh of $N_e=27130$ and with/without preconditioning (9-point pattern).}
\label{fig:ex1-3}
\end{figure}

\begin{figure}[ht!]
\centering
\subfigure[$k=0$]{
\includegraphics[width=0.22\textwidth]{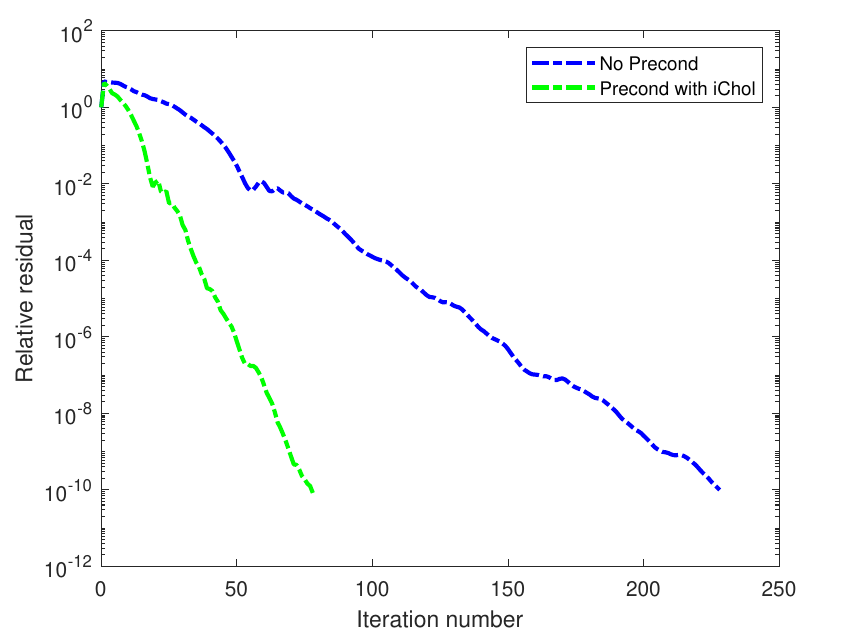}
}
\subfigure[$k=5$]{
\includegraphics[width=0.22\textwidth]{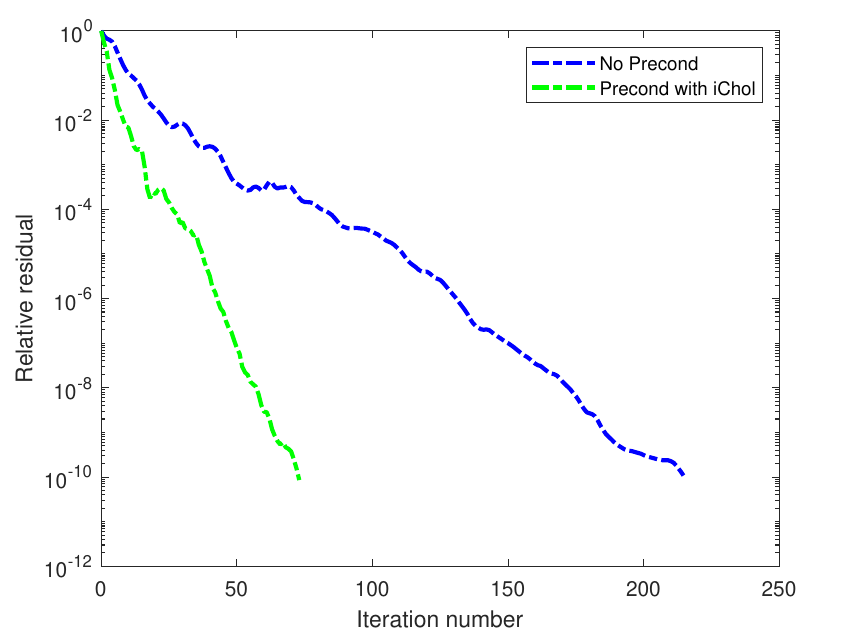}
}
\caption{Example~\ref{ex:1}. The CG convergence history is plotted for $k=0$ and 5 and $s=0.9$ with a non-adaptive
mesh of $N_e=27130$ and with/without preconditioning (5-point pattern).}
\label{fig:ex1-4}
\end{figure}

\begin{exam}\label{ex:4}
Next we consider the 3D version of (\ref{exam-0}) for $k = 0$ and $s = 0.5$.
Fig.~\ref{fig:ex4-3}(a) shows the convergence history in $L^2$ norm.
The solution error converges slightly lower than $\mathcal{O}(h)$
for fixed meshes and $\mathcal{O}(\bar{h}^{1.5})$ for adaptive meshes.
For this example, we take the 27-point pattern to build the preconditioner.  The CG convergence histories
shown in Fig.~\ref{fig:ex4-3}(b) and (c)
demonstrate the effectiveness of the preconditioner.
\qed
\end{exam}

\begin{figure}[ht!]
\centering
\subfigure[$s = 0.5$, soln error]{
\includegraphics[width=0.22\textwidth]{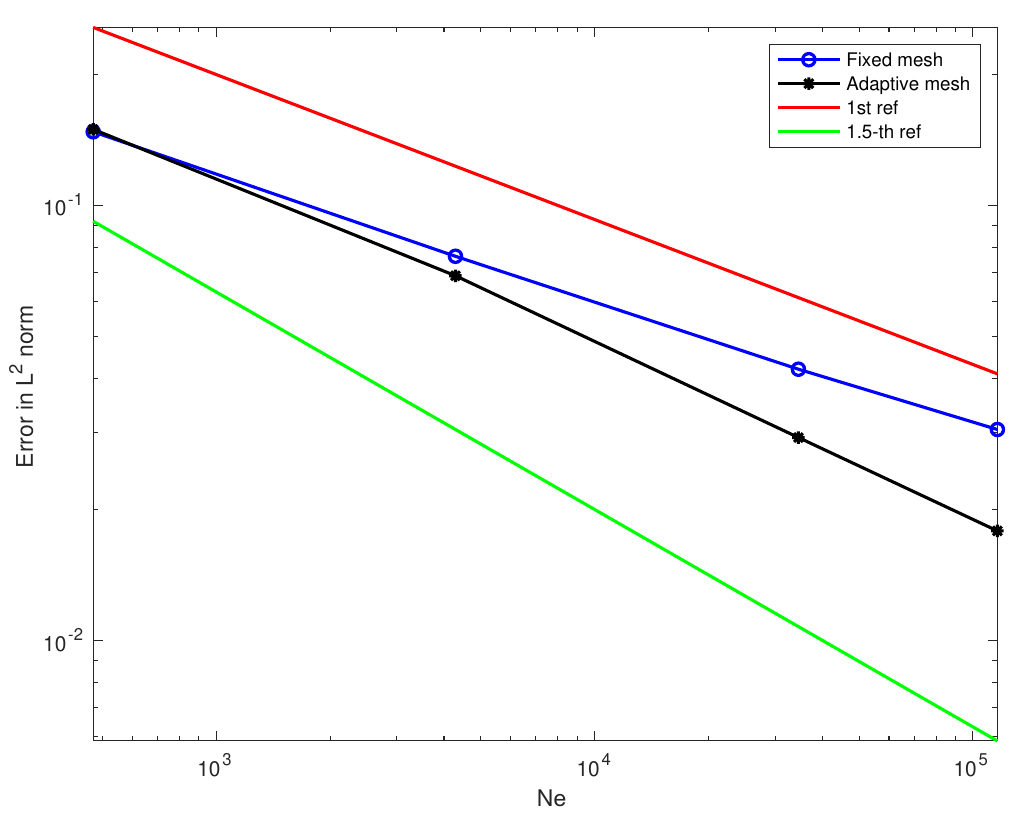}
}
\quad
\subfigure[$s = 0.5$, CG convg,]{
\includegraphics[width=0.22\textwidth]{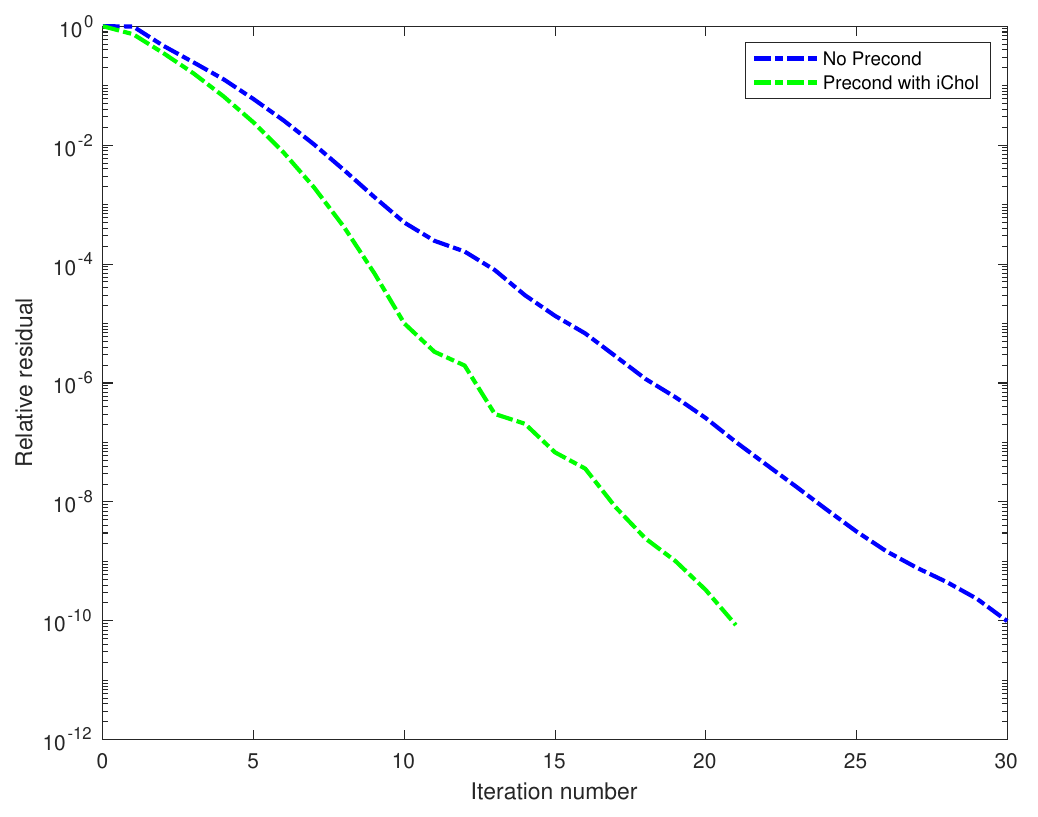}
}
\quad
\subfigure[$s = 0.9$, CG convg.]{
\includegraphics[width=0.22\textwidth]{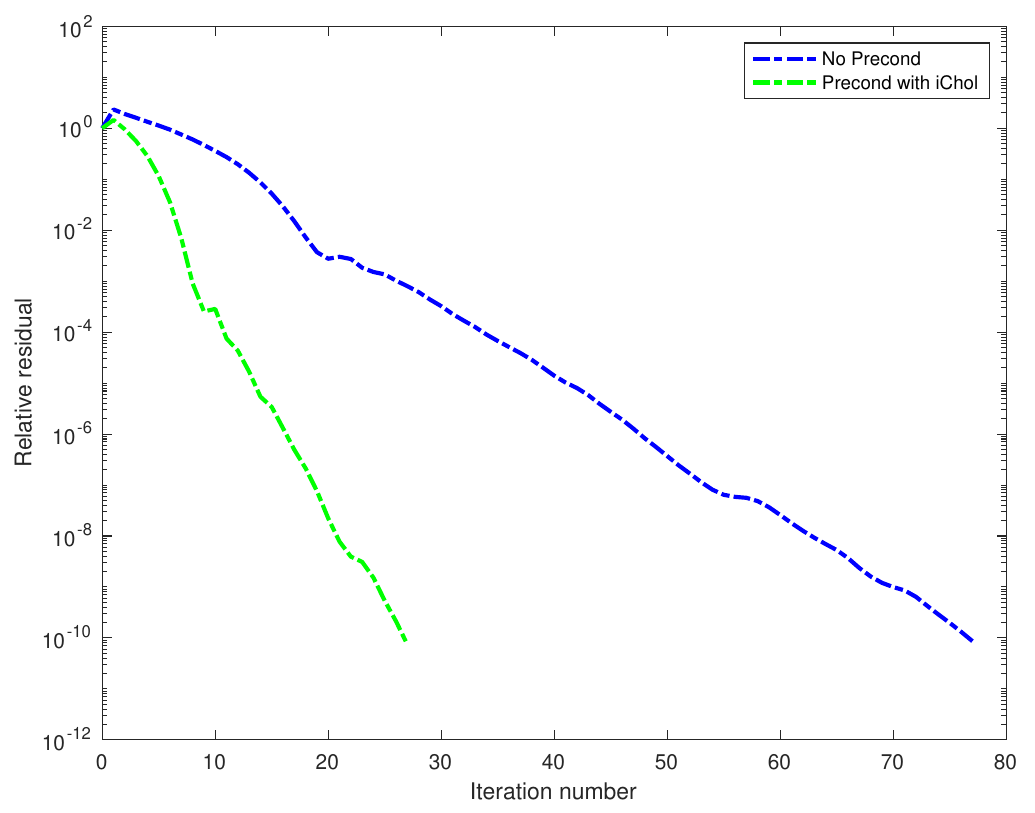}
}
\caption{Example~\ref{ex:4}. (a): The solution error as function of $N$ and (b) and (c): CG convergence histories for a fixed mesh of $N_e = 116054$ with and without preconditioning.}
\label{fig:ex4-3}
\end{figure}

\begin{exam}\label{ex:2}
This example is (\ref{FL-1}) with $f = 1$ and $\Omega$ as shown in Fig.~\ref{fig:gridoverlay-1} with $s = 0.75$.
The geometry of $\Omega$ is complex, with the wavering outside boundary and two holes inside.
An analytical exact solution is not available for this example. A computed solution with an adaptive mesh of $N_e=250948$
is used as the reference solution. Numerical results are shown in Fig.~\ref{fig:ex2-1}. The solution error in $L^2$ norm
is about $\mathcal{O}(h)$ for fixed meshes and $\mathcal{O}(\bar{h}^2)$ for adaptive meshes. This example demonstrates
that GoFD works well with complex geometries.
\qed
\end{exam}

\begin{figure}[ht!]
\centering
\subfigure[Adaptive mesh]{
\includegraphics[width=0.22\textwidth]{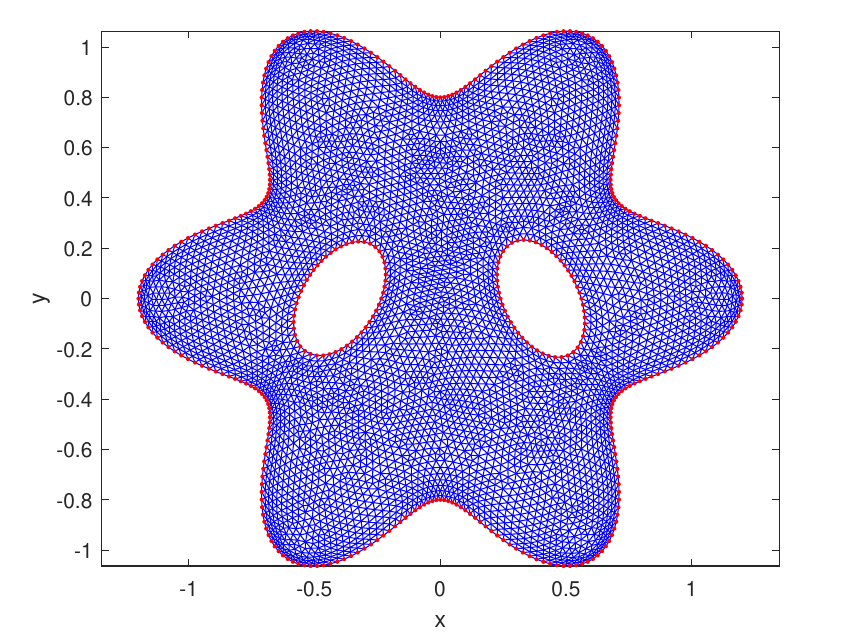}
}
\subfigure[Computed solution]{
\includegraphics[width=0.22\textwidth]{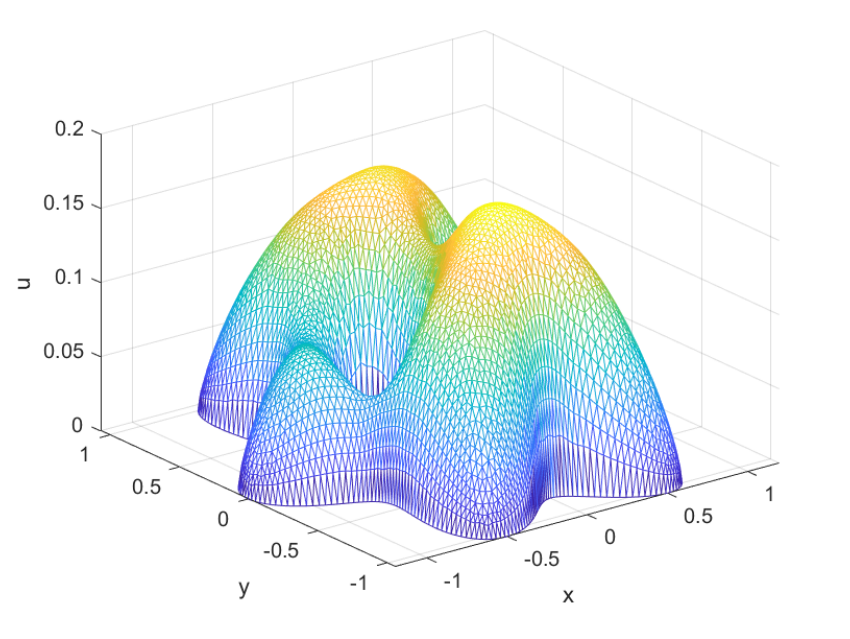}
}
\subfigure[Solution error]{
\includegraphics[width=0.22\textwidth]{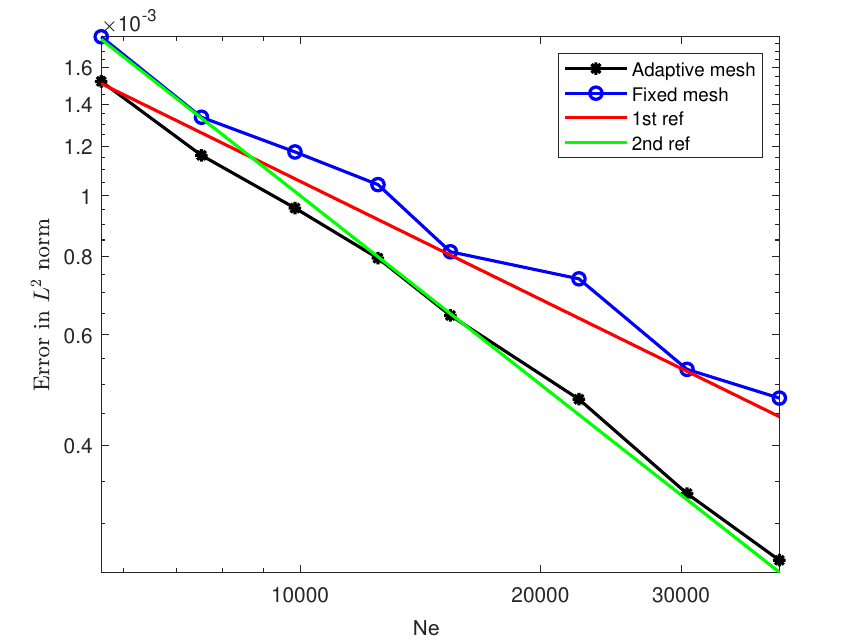}
}
\caption{Example~\ref{ex:2}.
(a): An adaptive mesh of $N_e=9850$, (b) the corresponding computed solution, and
(c) the solution error as function of $N_e$ for ${\color{magenta}s}=0.75$.}
\label{fig:ex2-1}
\end{figure}

\begin{exam}\label{ex:3}
This example is (\ref{FL-1}) with $f = 1$ and $\Omega$ being $L$-shaped with ${\color{magenta}s} = 0.5$.
An analytical exact solution is not available for this example. A computed solution obtained with an adaptive mesh of $N_e=417508$
is used as the reference solution.
Numerical results are shown in Figs.~\ref{fig:ex3-1}.
The solution error in $L^2$ norm is about $\mathcal{O}(h)$
for fixed meshes and $\mathcal{O}(\bar{h}^2)$ for adaptive meshes.
\qed
\end{exam}

\begin{figure}[ht!]
\centering
\subfigure[Adaptive mesh]{
\includegraphics[width=0.22\textwidth]{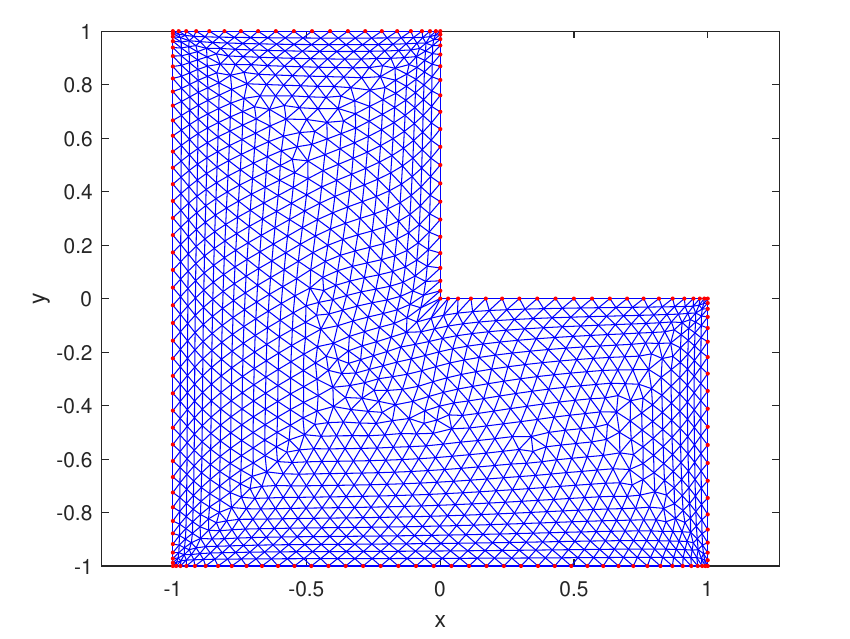}
}
\quad
\subfigure[Computed solution]{
\includegraphics[width=0.22\textwidth]{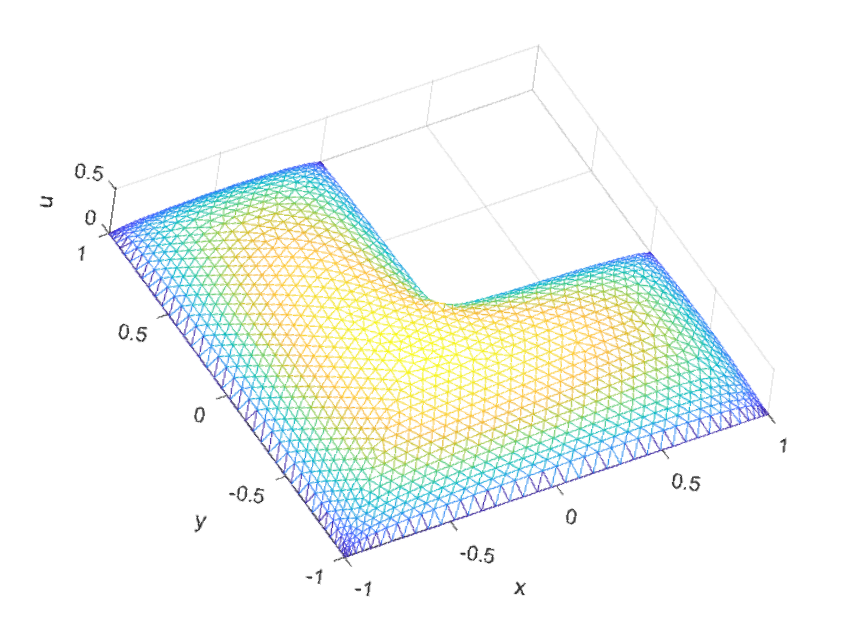}
}
\quad
\subfigure[Solution error]{
\includegraphics[width=0.22\textwidth]{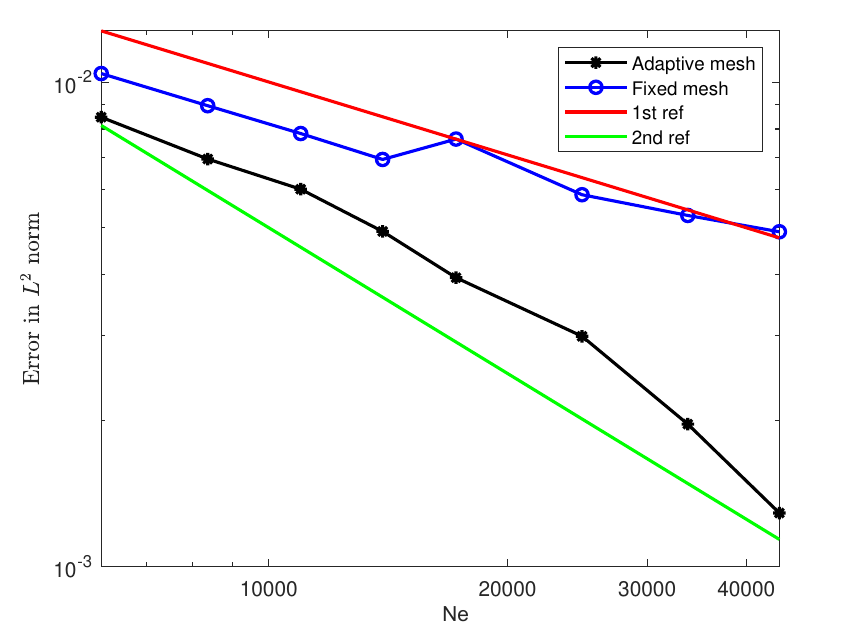}
}
\caption{Example~\ref{ex:3}. (a): An adaptive mesh of $N_e=2728$, (b) the corresponding computed solution, and
(c) the solution error as function of $N_e$ for $s=0.5$.}
\label{fig:ex3-1}
\end{figure}

\section{Conclusions and further comments}
\label{SEC:conclusions}

In the previous sections we have studied a grid-overlay finite difference method (GoFD) for the numerical approximation of
the fractional Laplacian on arbitrary bounded domains. The method uses an unstructured mesh
and an overlaying uniform grid and constructs the approximation
matrix $A_h$ (cf. (\ref{Ah-1})) based on the uniform-grid FD approximation $A_{\text{FD}}$ (cf. (\ref{T-1}) and (\ref{A-1}))
and the transfer matrix $I_h^{\text{FD}}$ from the unstructured mesh to the uniform grid.
The multiplication of $A_h$ with vectors can be carried out efficiently using FFT and sparse-matrix-vector multiplication.
A main result is Theorem~\ref{thm:Ah-1} stating that $A_h$ is similar to a symmetric and positive definite matrix
(and thus invertible) if $I_h^{\text{FD}}$ has full column rank and positive column sums.
A special choice of $I_h^{\text{FD}}$ is piecewise linear interpolation. Theorem~\ref{thm:IhFD} states that
the full column rank and positive column sums are guaranteed for this special choice if the spacing of the uniform grid
satisfies (\ref{hFD-1}). Stability and preconditioning for the resulting linear system have been discussed.

GoFD retains the efficient matrix-vector multiplication advantage of uniform-grid FD methods
for the fraction Laplacian while being able to work for domains with complex geometries.
Meanwhile, the method can readily be combined with existing adaptive mesh strategies due to its use of
unstructured meshes. We have discussed in Section~\ref{SEC:MMPDE} how to combine GoFD
with the MMPDE moving mesh method.

Numerical results have been presented for a selection of 1D, 2D and 3D examples. They have demonstrated
that GoFD is feasible and convergent and has a convergence order
of $\mathcal{O}(h^{\min(1,0.5+s)})$ in $L^2$ norm for fixed meshes.
This is consistent with observations known for existing uniform-grid FD and finite element methods.
With adaptive meshes, the method shows second-order convergence in 1D and 2D and close to
$\mathcal{O}(\bar{h}^{1+s})$ in 3D.
The numerical results have also demonstrated that the preconditioners based on the sparsity pattern of the Laplacian
(cf. Section~\ref{SEC:preconditioner}) are effective in terms of reducing the number of iterations required to reach
the commensurate accuracy.

Finally we comment that we have used unstructured simplicial meshes for $\Omega$ in this work.
The use of simplicial meshes makes it relatively simpler to prove  the full column rank of $I_h^{\text{FD}}$
and implement the transfer. However, it is not necessary to use simplicial meshes.
We can use any other boundary fitted meshes or even meshless points.  Particularly, we can take
$\mathcal{T}_h$ as a graded mesh. Moreover,
we can use data transfer schemes other than linear interpolation that has been considered in this work.
These are interesting topics worth future investigations.

\section*{Acknowledgments}
The first and second authors were supported in part by the University of Kansas General Research Fund FY23 and the National Natural Science Foundation of China through grant [12101509], respectively.

\bibliographystyle{siamplain}

\begin{thebibliography}{10}

\bibitem{Acosta2017}
{ G.~Acosta, F.~M. Bersetche, and J.~P. Borthagaray}, {\em A short {FE}
  implementation for a 2d homogeneous {D}irichlet problem of a fractional
  {L}aplacian}, Comput. Math. Appl., 74 (2017), 784--816.

\bibitem{Acosta201701}
{ G.~Acosta and J.~P. Borthagaray}, {\em A fractional {L}aplace equation:
  regularity of solutions and finite element approximations}, SIAM J. Numer.
  Anal., 55 (2017), 472--495.

\bibitem{Ainsworth-2017}
{ M.~Ainsworth and C.~Glusa}, {\em Aspects of an adaptive finite element
  method for the fractional {L}aplacian: a priori and a posteriori error
  estimates, efficient implementation and multigrid solver}, Comput. Methods
  Appl. Mech. Engrg., 327 (2017), 4--35.

\bibitem{Ainsworth-2018}
{ M.~Ainsworth and C.~Glusa}, {\em Towards an efficient finite element
  method for the integral fractional {L}aplacian on polygonal domains}, in
  Contemporary computational mathematics -- A celebration of the 80th birthday
  of {I}an {S}loan. {V}ol. 1, 2, Springer, Cham, 2018, 17--57.

\bibitem{Antil-2022}
{ H.~Antil, T.~Brown, R.~Khatri, A.~Onwunta, D.~Verma, and M.~Warma}, {\em
  Chapter 3 - {Optimal} control, numerics, and applications of fractional
  {PDEs}}, Handbook of Numerical Analysis, 23 (2022), 87--114.

\bibitem{Antil-2021}
{ H.~Antil, P.~Dondl, and L.~Striet}, {\em Approximation of integral
  fractional {L}aplacian and fractional {PDE}s via sinc-basis}, SIAM J. Sci.
  Comput., 43 (2021), A2897--A2922.

\bibitem{Bahouri-2011}
{ H.~Bahouri, J.-Y. Chemin, and R.~Danchin}, {\em Fourier analysis and
  nonlinear partial differential equations}, vol.~343 of Grundlehren der
  mathematischen Wissenschaften [Fundamental Principles of Mathematical
  Sciences], Springer, Heidelberg, 2011.

\bibitem{Bonito2019}
{ A.~Bonito, W.~Lei, and J.~E. Pasciak}, {\em Numerical approximation of the
  integral fractional {L}aplacian}, Numer. Math., 142 (2019), 235--278.

\bibitem{Borthagaray-2018}
{ J.~P. Borthagaray, L.~M. Del~Pezzo, and S.~Mart\'{\i}nez}, {\em Finite
  element approximation for the fractional eigenvalue problem}, J. Sci.
  Comput., 77 (2018), 308--329.

\bibitem{Chan-1996}
{ R.~H. Chan and M.~K. Ng}, {\em Conjugate gradient methods for {T}oeplitz
  systems}, SIAM Rev., 38 (1996), 427--482.

\bibitem{DuNing2019}
{ N.~Du, H.-W. Sun, and H.~Wang}, {\em A preconditioned fast finite
  difference scheme for space-fractional diffusion equations in convex
  domains}, Comput. Appl. Math., 38 (2019), Paper No. 14.

\bibitem{Du2019}
{ Q.~Du, L.~Ju, and J.~Lu}, {\em A discontinuous {G}alerkin method for
  one-dimensional time-dependent nonlocal diffusion problems}, Math. Comp., 88
  (2019), 123--147.

\bibitem{Du2020}
{ Q.~Du, L.~Ju, J.~Lu, and X.~Tian}, {\em A discontinuous {G}alerkin method
  with penalty for one-dimensional nonlocal diffusion problems}, Commun. Appl.
  Math. Comput., 2 (2020), 31--55.

\bibitem{Duo-2018}
{ S.~Duo, H.~W. van Wyk, and Y.~Zhang}, {\em A novel and accurate finite
  difference method for the fractional {L}aplacian and the fractional {P}oisson
  problem}, J. Comput. Phys., 355 (2018), 233--252.

\bibitem{Faustmann2022}
{ M.~Faustmann, M.~Karkulik, and J.~M. Melenk}, {\em Local convergence of
  the {FEM} for the integral fractional {L}aplacian}, SIAM J. Numer. Anal., 60
  (2022), 1055--1082.

\bibitem{Filon-1928}
{ L.~N.~G. Filon}, {\em On a quadrature formula for trigonometric
  integrals}, Model. Anal. Inf. Sist., 49 (1928), 38--47.

\bibitem{Hao2021}
{ Z.~Hao, Z.~Zhang, and R.~Du}, {\em Fractional centered difference scheme
  for high-dimensional integral fractional {L}aplacian}, J. Comput. Phys., 424
  (2021), Paper No. 109851.

\bibitem{HJ1985}
{ R.~A. Horn and C.~A. Johnson}, {\em Matrix Analysis}, Cambridge University
  Press, Cambridge, London, 1985.

\bibitem{HK2015}
{ W.~Huang and L.~Kamenski}, {\em A geometric discretization and a simple
  implementation for variational mesh generation and adaptation}, J. Comput.
  Phys., 301 (2015), 322--337.

\bibitem{HRR94a}
{ W.~Huang, Y.~Ren, and R.~D. Russell}, {\em Moving mesh partial
  differential equations ({MMPDEs}) based upon the equidistribution principle},
  SIAM J. Numer. Anal., 31 (1994), 709--730.

\bibitem{HR11}
{ W.~Huang and R.~D. Russell}, {\em Adaptive Moving Mesh Methods}, Springer,
  New York, 2011.
\newblock Applied Mathematical Sciences Series, Vol. 174.

\bibitem{HS03}
{ W.~Huang and W.~Sun}, {\em Variational mesh adaptation {II}: error
  estimates and monitor functions}, J. Comput. Phys., 184 (2003), 619--648.

\bibitem{Huang2014}
{ Y.~Huang and A.~Oberman}, {\em Numerical methods for the fractional
  {L}aplacian: a finite difference-quadrature approach}, SIAM J. Numer. Anal.,
  52 (2014), 3056--3084.

\bibitem{Huang2016}
{ Y.~Huang and A.~Oberman}, {\em Finite difference methods for fractional
  laplacians}, arXiv preprint arXiv:1611.00164,  (2016).

\bibitem{Llic-2005}
{ M.~Ilic, F.~Liu, I.~Turner, and V.~Anh}, {\em Numerical approximation of a
  fractional-in-space diffusion equation. {I}}, Fract. Calc. Appl. Anal., 8
  (2005), 323--341.

\bibitem{LiHuiyuan2022}
{ H.~Li, R.~Liu, and L.-L. Wang}, {\em Efficient hermite spectral-{G}alerkin
  methods for nonlocal diffusion equations in unbounded domains}, Numer. Math.
  Theory Methods Appl., 15 (2022), 1009--1040.

\bibitem{Lischke-2020}
{ A.~Lischke, G.~Pang, M.~Gulian, and et~al.}, {\em What is the fractional
  {L}aplacian? {A} comparative review with new results}, J. Comput. Phys., 404
  (2020), 109009.

\bibitem{Ying-2020}
{ V.~Minden and L.~Ying}, {\em A simple solver for the fractional laplacian
  in multiple dimensions}, SIAM J. Sci. Comput., 42 (2020), A878--A900.

\bibitem{Nevskii-2018}
{ M.~V. Nevski\u{\i}}, {\em On some problems for a simplex and a ball in
  {$\Bbb R^n$}}, Model. Anal. Inf. Sist., 25 (2018), 680--691.

\bibitem{Ortigueira2006}
{ M.~D. Ortigueira}, {\em Riesz potential operators and inverses via
  fractional centred derivatives}, Int. J. Math. Math. Sci.,  (2006), Art.
  ID 48391.

\bibitem{Ortigueira2008}
{ M.~D. Ortigueira}, {\em Fractional central differences and derivatives},
  J. Vib. Control, 14 (2008), 1255--1266.

\bibitem{Pang-2012}
{ H.-K. Pang and H.-W. Sun}, {\em Multigrid method for fractional diffusion
  equations}, J. Comput. Phys., 231 (2012), 693--703.

\bibitem{Ros-Oton-2014}
{ X.~Ros-Oton and J.~Serra}, {\em The {D}irichlet problem for the fractional
  {L}aplacian: regularity up to the boundary}, J. Math. Pures Appl., 101
  (2014), 275--302.

\bibitem{Saito-2015}
{ N.~Saito and G.~Zhou}, {\em Analysis of the fictitious domain method with
  an {$L^2$}-penalty for elliptic problems}, Numer. Funct. Anal. Optim., 36
  (2015), 501--527.

\bibitem{Song2017}
{ F.~Song, C.~Xu, and G.~E. Karniadakis}, {\em Computing fractional
  {L}aplacians on complex-geometry domains: algorithms and simulations}, SIAM
  J. Sci. Comput., 39 (2017), A1320--A1344.

\bibitem{Sunjing2021}
{ J.~Sun, D.~Nie, and W.~Deng}, {\em Algorithm implementation and numerical
  analysis for the two-dimensional tempered fractional {L}aplacian}, BIT, 61
  (2021), 1421--1452.

\bibitem{Tian2013}
{ X.~Tian and Q.~Du}, {\em Analysis and comparison of different
  approximations to nonlocal diffusion and linear peridynamic equations}, SIAM
  J. Numer. Anal., 51 (2013), 3458--3482.

\bibitem{Wang2012}
{ H.~Wang and T.~S. Basu}, {\em A fast finite difference method for
  two-dimensional space-fractional diffusion equations}, SIAM J. Sci. Comput.,
  34 (2012), A2444--A2458.

\bibitem{YangQianqian2011}
{ Q.~Yang, I.~Turner, F.~Liu, and M.~Ili\'{c}}, {\em Novel numerical methods
  for solving the time-space fractional diffusion equation in two dimensions},
  SIAM J. Sci. Comput., 33 (2011), 1159--1180.

\end{thebibliography}

\end{document}